\documentclass{amsart}
\usepackage[T1]{fontenc}
\usepackage{amsmath,amsfonts,amsthm,amssymb,indentfirst,epic,url,graphics,needspace}

\newtheorem{theorem}{Theorem}
\newtheorem{lemma}[theorem]{Lemma}

\newtheorem{proposition}[theorem]{Proposition}
\newtheorem{definition}[theorem]{Definition}
\newtheorem{question}[theorem]{Question}

\numberwithin{equation}{section}
\numberwithin{theorem}{section}

\newcommand{\Z}{\mathbb{Z}}
\renewcommand{\r}{\mathrm}

\newcommand{\langl}{\begin{picture}(5.1,10)
\put(1.1,3.3){\rotatebox{60}{\line(1,0){5.5}}}
\put(1.1,3.3){\rotatebox{300}{\line(1,0){5.5}}}
\end{picture}}

\newcommand{\rangl}{\begin{picture}(5,10)
\put(.9,3.3){\rotatebox{120}{\line(1,0){5.5}}}
\put(.9,3.3){\rotatebox{240}{\line(1,0){5.5}}}
\end{picture}}

\begin{document}

\title%
{On core quandles of groups}
\thanks{
Archived at \url{http://arxiv.org/abs/2006.00641}\,.
After publication, any updates, errata, related references,
etc., found will be recorded at
\url{http://math.berkeley.edu/~gbergman/papers/}.
}

%         np.gps

\author{George M. Bergman}
\address{University of California\\
Berkeley, CA 94720-3840, USA}
\email{gbergman@math.berkeley.edu}

\begin{abstract}
We review the definition of a {\em quandle,} and in particular
of the {\em core quandle} $\r{Core}(G)$ of a group $G,$ which
consists of the underlying set of $G,$ with
the binary operation $x\lhd y = x y^{-1} x.$
This is an {\em involutory} quandle, i.e., satisfies the identity
$x\lhd (x\lhd y) = y$ in addition to the other identities
defining a quandle.

{\em Trajectories} $(x_i)_{i\in\Z}$ in groups and in involutory
quandles (in the former context, sequences of the form
$x_i = x z^i$ $(x,z\in G)$ among other characterizations;
in the latter, sequences satisfying $x_{i+1}= x_i\lhd\,x_{i-1})$
are examined.
A family of necessary conditions for an involutory quandle $Q$ to be
embeddable in the core quandle of a group is noted.
Some implications are established
between identities holding in groups and in their core quandles.
Upper and lower bounds are obtained on the number of elements needed
to generate the quandle $\r{Core}(G)$ for $G$ a finitely
generated group.
Several questions are posed.
\end{abstract}
\maketitle
% - - - - - - - - - - - - - - - - - - - - - - - - - - - - - -

\section{Background}\label{S.bkgd}

The concept of quandle arose in knot theory, as a way of studying
knot groups in terms of their conjugation operation.
If for $G$ a group one defines
\begin{equation}\begin{minipage}[c]{35pc}\label{d.conj}
$x\lhd y\ =\ x\,y\,x^{-1}$ $(x,y\in G),$
\end{minipage}\end{equation}
and denotes by $Q$ the underlying set of $G$ with this operation,
one finds that
\begin{equation}\begin{minipage}[c]{35pc}\label{d.idpt}
For all $x\in Q,$ $x\lhd x\ =\ x.$
\end{minipage}\end{equation}
\begin{equation}\begin{minipage}[c]{35pc}\label{d.bij}
For all $x\in Q,$ the map $y\mapsto x\lhd y$ is a
bijection $Q\to Q.$
\end{minipage}\end{equation}
\begin{equation}\begin{minipage}[c]{35pc}\label{d.end}
For all $x,y,z\in Q,$ $x\lhd(y\lhd z)\ =\ (x\lhd y)\lhd(x\lhd z).$
In other words, for all $x\in Q,$ the map $y\mapsto x\lhd y$ is an
endomorphism of $( Q,\lhd).$
\end{minipage}\end{equation}

More generally, for any integer $d,$~\eqref{d.idpt}-\eqref{d.end} hold
for the operation on a group $G$ given by
\begin{equation}\begin{minipage}[c]{35pc}\label{d.conj_d}
$x\lhd y\ =\ x^d\,y\ x^{-d}$ $(x,y\in G).$
\end{minipage}\end{equation}

There is one more derived operation on groups $G$
for which~\eqref{d.idpt}-\eqref{d.end} hold,
which has been studied less
(though it too has been used in knot theory~\cite{top_core}),
but is the main subject of this note; namely
\begin{equation}\begin{minipage}[c]{35pc}\label{d.core}
$x\lhd y\ =\ x\ y^{-1}\,x$ $(x,y\in G).$
\end{minipage}\end{equation}

This last operation also satisfies the identity
\begin{equation}\begin{minipage}[c]{35pc}\label{d.inv}
For all $x,y\in Q,$ $x\lhd(x\lhd y)\ =\ y.$
\end{minipage}\end{equation}

Here is the terminology used for the above sorts of structures
(though formalizations and notations vary).

\begin{definition}[cf.\ {\cite{qndl}, \cite{JSC}, \cite{Wiki_quandle}}]\label{D.qndl}
A {\em quandle} is a set $Q$ given with a binary operation
$\lhd\!:\ Q^2\to Q$ satisfying~\eqref{d.idpt},~\eqref{d.bij}
and~\eqref{d.end}.
A quandle is said to be {\em involutory} if it also
satisfies~\eqref{d.inv}.

If $G$ is a group, then the quandle given by the
underlying set of $G$ with
the operation~\eqref{d.conj} is denoted $\r{Conj}(G),$ while the
\textup{(}involutory\textup{)} quandle given by the same set with the
operation~\eqref{d.core} is denoted $\r{Core}(G).$
\end{definition}

(Instead of ``involutory'', the form ``involutive'' is sometimes used.)

The above terminology was introduced, and many results on
quandles developed, in~\cite{qndl}.
(See the Introduction to that paper for notes on earlier
literature that considered cases of the concept, and
cf.\ also~\cite{Wiki_quandle}.)
It is shown in \cite[Proposition 3.1]{bi-q}
that the operations~\eqref{d.conj_d} for all $d\in\Z,$ and the lone
additional operation~\eqref{d.core},
are in fact the only derived group operations
that give quandle structures on the underlying
sets of all groups~$G.$

My own path to the construction that I subsequently
learned is called $\r{Core}(G)$ involved
conditions on a group $G$ related to one-sided orderability.
Here, sequences of group elements of the form
\begin{equation}\begin{minipage}[c]{35pc}\label{d.traj}
$(x\,z^i)_{i\in\Z}$ $(x,z\in G)$
\end{minipage}\end{equation}
seemed important.
(I will not mention orderability after this paragraph,
but for those conversant with the subject, a condition on $G$
weaker than one-sided orderability, called
``locally invariant orderability'' \cite{LIO},
is equivalent to the existence of a total ordering on the
underlying set of $G$ under which each
sequence~\eqref{d.traj} is either monotone increasing, monotone
decreasing, or decreasing up to a certain point and
increasing thereafter.
An intermediate condition is, of course, the
existence of an ordering under which every sequence~\eqref{d.traj} is
monotone increasing or decreasing.
Whether one or the other of the implications from
one-sided orderability to the latter property to
the former is reversible, is not known.)
Calling a sequence of the
form~\eqref{d.traj} in a group a ``trajectory'', one sees
that trajectories can also be characterized as the
sequences $(x_i)_{i\in\Z}$ such that for all $i,$
$x_{i+1} = x_i\,x_{i-1}^{-1}\,x_i;$ equivalently,
$x_{i-1} = x_i\,x_{i+1}^{-1}\,x_i.$
Though I never got anywhere with using them to study orderability,
I found the concept of trajectory
and the properties of the operation~\eqref{d.core} that underlies
it intriguing.
I eventually learned that the kind of
structure I was looking at had already been named, as described above.

A disadvantage of condition~\eqref{d.bij} of Definition~\ref{D.qndl}
is that it is not expressed by identities.
To do that, one can introduce a second binary operation
(written $y\rhd x,$ or $x\lhd^{-1} y)$
which inverts the effect of $x\lhd\,.$
This gives another commonly used formulation of the concept of quandle.
However, I have used Definition~\ref{D.qndl} here
because in the case of involutory quandles,~\eqref{d.bij} is
implied by~\eqref{d.inv}, and hence can be dropped,
with no additional operation needed.

Let us also note that in the presence of~\eqref{d.inv},
the identity of~\eqref{d.end} is equivalent to the identity
gotten by replacing $z$ everywhere in it by $x\lhd z,$
applying~\eqref{d.inv} to the resulting occurrence of
$x\lhd(x\lhd z),$ and interchanging the two sides:
\begin{equation}\begin{minipage}[c]{35pc}\label{d.end2}
For all $x,y,z\in Q,$
$(x\lhd y)\lhd z\ =\ x\lhd(y\lhd(x\lhd z)).$
\end{minipage}\end{equation}
This formula will prove useful in that it allows us to
reduce any $\!\lhd\!$-expression to one in which parentheses
are clustered to the right.

Summarizing, we have

\begin{lemma}\label{L.redef}
An involutory quandle can be characterized as a set
$Q$ given with a binary operation $\lhd$
satisfying~\eqref{d.idpt},~\eqref{d.inv} and~\eqref{d.end2}.\qed
\end{lemma}

We note for later reference the easily checked result:

\begin{lemma}\label{L.auts}
If $G$ is a group, then the following sorts of permutations of the
underlying set of $G,$ defined in terms of the group
structure of $G,$ give automorphisms of the
involutory quandle $\r{Core}(G).$\\[.2em]
\textup{(i)} \ For every $g\in G,$ the map $x\mapsto x\,g.$\\[.2em]
\textup{(ii)}\,\ For every $h\in G,$ the map $x\mapsto h\,x.$\\[.2em]
\textup{(iii)} The map $x\mapsto x^{-1}.$ \qed
\end{lemma}

\section{A normal form for free involutory quandles}\label{S.nml_form}

Let us now prove

\begin{theorem}\label{T.ids}
The identities satisfied by the derived operation~\eqref{d.core} on
all groups are precisely the consequences
of~\eqref{d.idpt},~\eqref{d.inv} and~\eqref{d.end2}.

Any word in a set of symbols $X$ and
the operation-symbol $\lhd$ can be reduced, using these identities,
to a unique expression
\begin{equation}\begin{minipage}[c]{35pc}\label{d.x_0dots}
$x_0\lhd(x_1\lhd( \dots \lhd(x_{n-1}\lhd x_n)\dots))$\quad
\textup{(}with parentheses clustered on the right\textup{)},
where all $x_i\in X,$ and no two successive arguments
$x_i,\ x_{i+1}$ are the same.
\end{minipage}\end{equation}

Thus, the expressions~\eqref{d.x_0dots} give a normal form
for elements of the free involutory quandle on $X.$
\end{theorem}

\begin{proof}
The verification of~\eqref{d.idpt},~\eqref{d.inv} and~\eqref{d.end2}
for the operation~\eqref{d.core} is immediate.
Postponing the claim that those three identities
imply all identities satisfied by that derived operation,
we note that given any word in $\lhd$ and symbols from
$X,$~\eqref{d.end2} can indeed be used recursively to reduce it to
one in which parentheses are clustered to the right.
(To see formally that recursive application
of~\eqref{d.end2} must terminate, let us define
the {\em implicit length} of a $\!\lhd\!$-word $w$ in symbols from $X$
by letting the implicit length of each $x\in X$ be $1,$
and the implicit length of a word $w_1\lhd w_2$ be
the implicit length of $w_2,$ plus twice the implicit length of $w_1.$
We find that any application of~\eqref{d.end2} to
a subword of a $\!\lhd\!$-word leaves the word's implicit length
unchanged, but increases its length (number of occurrences
of variable-symbols); so, since the length is
bounded above by the implicit length, the process must terminate.)
We can, next, use \eqref{d.inv} recursively to eliminate
cases where $x_i=x_{i+1}$ for $i<n-1,$ and, finally, use~\eqref{d.idpt}
recursively to eliminate cases where $x_{n-1}=x_n,$
giving a word of the form~\eqref{d.x_0dots}.

To show uniqueness, note that given elements
$x_0,\dots,x_n$ in a group $G,$
the expression in~\eqref{d.x_0dots}, evaluated
in $\r{Core}(G),$ describes the group element
\begin{equation}\begin{minipage}[c]{35pc}\label{d.x_0dots_gp}
$x_0\ x_1^{-1}\ x_2 \dots\,x_{n-1}^{\mp 1}\,x_n^{\pm 1}\,
x_{n-1}^{\mp 1}\,\dots\,x_2\ x_1^{-1}\,x_0.$
\end{minipage}\end{equation}
Now if we take for $G$ the free group on the elements of $X,$
then by the condition in~\eqref{d.x_0dots} that
no two successive $x_i$ be equal,~\eqref{d.x_0dots_gp}
is a reduced word in that free group,
whose value in that group determines the sequence $x_0,\dots,x_n.$
So starting with an arbitrary $\!\lhd\!$-word in the elements
of $X,$ any two expressions of the form~\eqref{d.x_0dots}
obtainable from it using~\eqref{d.idpt},~\eqref{d.inv}
and~\eqref{d.end2} must be the same, which is the desired
uniqueness statement.

Returning to the claim whose verification we postponed,
suppose $u=v$ is an identity satisfied by $\lhd$ in all groups.
Applying~\eqref{d.idpt},~\eqref{d.inv} and~\eqref{d.end2}
as above, we can reduce $u$ and $v$
to words of the form~\eqref{d.x_0dots}.
Since we have assumed $u=v$ to hold identically in core quandles of
groups, the reduced expressions
have, in particular, the same value in the core quandle of the
free group on $X;$ so by the above uniqueness
result, they must be the same.
So the equality
$u=v$ is indeed a consequence of~\eqref{d.idpt},~\eqref{d.inv}
and~\eqref{d.end2}.
\end{proof}

This immediately yields the first assertion of

\begin{proposition}\label{P.free}
Let $X$ be a nonempty set.
Then the elements of the free group $\langl X\rangl$ of the
form~\eqref{d.x_0dots_gp}, i.e., the symmetric reduced group words
of odd length in the elements of $X,$
in which the exponents alternate between $+1$ and
$-1,$ starting with the former, form a subquandle of
$\r{Core}(\langl X\rangl)$ which is a free involutory quandle on $X.$

On the other hand, fixing an element $y$ of $X,$ the set of
{\em all} symmetric reduced group words in $X-\{y\}$
\textup{(}including the empty word $1,$ and with no
condition of odd length or alternating exponents\textup{)} forms
a subquandle of $\r{Core}(\langl X-\{y\}\rangl)$
which is a free involutory quandle on the set $X-\{y\}\cup\{1\}.$
\end{proposition}

\begin{proof}
The assertion of the first paragraph follows from the
proof of Theorem~\ref{T.ids}.
To deduce the second paragraph, note that by Lemma~\ref{L.auts}(i), the
endomap
\begin{equation}\begin{minipage}[c]{35pc}\label{d.w_to_wy-1}
$w\ \mapsto\ wy^{-1}$
\end{minipage}\end{equation}
of the underlying set of
$\langl X\rangl$ is an automorphism of $\r{Core}(\langl X\rangl);$
hence a free involutory subquandle of $\r{Core}(\langl X\rangl)$
is also generated by the elements $xy^{-1}$ $(x\in X).$
Now the elements $xy^{-1}$ $(x\in X-\{y\})$ form a free
generating set of a subgroup of $\langl X\rangl,$
and that subgroup, of course, also contains $yy^{-1}=1,$
hence it contains our translated free involutory quandle.
If we map this free subgroup into $\langl X\rangl$ by sending each free
generator $xy^{-1}$ $(x\in X-\{y\})$ to $x,$ we get an isomorphism of
that subgroup with $\langl X-\{y\}\rangl,$
and we see that the isomorphism of free involutory
quandles given by our translation followed by this map
fixes each member of $X-\{y\},$ and sends $y$ to $1.$

I now claim that the result of applying this quandle isomorphism
to all reduced words of the form~\eqref{d.x_0dots_gp} is the set
of all symmetric reduced group words in $X-\{y\}.$
Indeed, given a symmetric reduced word $u$ in $X-\{y\},$
we may obtain a $w$ as
in~\eqref{d.x_0dots_gp} which maps to it as follows.
On the one hand, if $u$ begins (and hence ends) with a symbol having
exponent $-1,$ append a $y$ at the beginning and a $y$ the end.
Further, wherever $u$ has two successive variable-symbols with
the same exponent $+1$ or $-1,$ insert a $y$ with the opposite
exponent between them.
(In particular, if $u$ has positive {\em even} length, a $y$
or $y^{-1}$ is inserted in the middle.)
Finally, if $u=1,$ let $w=y.$

That the resulting word $w$ has the form~\eqref{d.x_0dots_gp},
and is mapped to $u$ under the isomorphism described, is immediate.
\end{proof}

In the proof of Theorem~\ref{T.ids}, we
used the identity~\eqref{d.end2} to bring words
to a form with parentheses clustered to the right.
It is helpful to note a consequence
of that identity (of which~\eqref{d.end2} itself is the $n=2$ case),
which describes how such a right-clustered expression acts by $\lhd.$
\begin{equation}\begin{minipage}[c]{35pc}\label{d.end3}
$x_1\lhd(x_2\lhd( \dots \lhd(x_{n-1}\lhd x_n)\dots))\lhd y\ =\\
\hspace*{2em}x_1\lhd(x_2\lhd( \dots \lhd(x_{n-1}\lhd (x_n
\lhd(x_{n-1} \lhd( \dots \lhd(x_2\lhd(x_1\lhd y))\dots))))\dots)).$
\end{minipage}\end{equation}

This is straightforward to check in a quandle
of the form $\r{Core}(G),$ using the definition~\eqref{d.core},
and the fact that a quandle expression~\eqref{d.x_0dots}
corresponds to the group expression~\eqref{d.x_0dots_gp}.
From this, the same result for a general involutory quandle $Q$
follows by the first statement of Theorem~\ref{T.ids}, since
that says that identities holding
in every quandle $\r{Core}(G)$ hold in {\em all} involutory quandles.
Alternatively, one can
prove~\eqref{d.end3} inductively from~\eqref{d.end2}.

\section{Trajectories in involutory quandles}\label{S.traj}

As mentioned, I was led to the topic of this
note by thinking about {\em trajectories} in groups $G,$ that is,
sequences of the form $(x\,z^i)_{i\in\Z}$ $(x,z\in G),$ equivalently,
sequences $(x_i)_{i\in\Z}$ satisfying $x_{i+1}=x_i\,x_{i-1}^{-1}\,x_i.$
(It is easy to see that these can also be described
as sequences of the form $(w^i\,x)_{i\in\Z}$ $(w,x\in G),$
or more generally, as sequences of the form $(x\,y^i\,x')$
$(x,y,x'\in G).$
In the first two sorts of expression, $x$ and $z,$
respectively $w$ and $x,$ are clearly unique, while in the last,
$x,$ $y$ and $x'$ can be replaced by $xu,$ $u^{-1}yu,$ $u^{-1}x'$
for any $u\in G.)$
Noting the form such a sequence takes in the quandle $Q=\r{Core}(G),$
and abstracting to general involutory quandles, we make

\begin{definition}\label{D.traj2}
If $Q$ is an involutory quandle, then a sequence
$(x_i)_{i\in\Z}$ of elements
of $Q$ will be called a {\em trajectory} in $Q$ if it satisfies
\begin{equation}\begin{minipage}[c]{35pc}\label{d.traj+1}
$x_{i+1} \ = \ x_i\,\lhd\, x_{i-1}$ \ for all $i\in\Z,$
\end{minipage}\end{equation}
equivalently \textup{(}as one sees
by applying $x_i\lhd$ to both sides of~\eqref{d.traj+1}\textup{)}, if
\begin{equation}\begin{minipage}[c]{35pc}\label{d.traj-1}
$x_{i-1} \ = \ x_i\,\lhd\, x_{i+1}$ \ for all $i\in\Z.$
\end{minipage}\end{equation}

In particular, if $G$ is a group, the sequences of elements
that are trajectories in the involutory quandle $\r{Core}(G)$
are precisely those that are trajectories in the group-theoretic
sense in $G.$
\end{definition}

(What we are calling trajectories in an involutory
quandle are roughly what
are called ``geodesics'' in~\mbox{\cite[\S11]{qndl}}; except that
there, after involutory quandles have been defined,
the concept of a set $Q$ with geodesics is
defined independently, and then shown to yield a
structure of involutory quandle on~$Q.)$

If $(x_i)_{i\in\Z}$ is a trajectory in a group, then letting
$x=x_0,$ $y=x_1,$ we see that
\begin{equation}\begin{minipage}[c]{35pc}\label{d.x_n}
$x_i\ =\ x\,(x^{-1} y)^i.$
\end{minipage}\end{equation}
Given a trajectory $(x_i)_{i\in\Z}$ in an involutory quandle
$Q,$ we can likewise, using~\eqref{d.traj+1} and~\eqref{d.traj-1},
write all $x_i$ in terms of $x=x_0$ and $y=x_1,$ though the
description is not as simple as~\eqref{d.x_n}.
Let me just list the forms
of $x_{-3}$ through $x_4,$ from which the pattern is clear.
\vspace{.2em}
\begin{equation}\begin{minipage}[c]{35pc}\label{d.x_n2}
$\hspace*{2em}\dots\\
x_{-3}=\ x\lhd(y\lhd(x\lhd y))\\
x_{-2}=\ x\lhd(y\lhd x)\\
x_{-1}=\ x\lhd y\\
x_0\ \ =\ x\\
x_1\ \ =\ y\\
x_2\ \ =\ y\lhd x\\
x_3\ \ =\ y\lhd(x\lhd y)\\
x_4\ \ =\ y\lhd(x\lhd(y\lhd x))\\
\hspace*{2em}\dots$
\end{minipage}\end{equation}
\vspace{.2em}

Again, this family of formulas can be proved either by
establishing them in quandles $\r{Core}(G),$ where they are
translations of the corresponding cases of~\eqref{d.x_n},
or by direct computation (in which case~\eqref{d.end3} is helpful).
A generalization of~\eqref{d.traj+1} and~\eqref{d.traj-1},
which can likewise be shown in either of these ways to hold in
trajectories $(x_i)_{i\in\Z}$ in involutory quandles, is
\begin{equation}\begin{minipage}[c]{35pc}\label{d.traj.i,j}
$x_i\lhd x_j\ =\ x_{2i-j}$ \ for all $i,j\in\Z.$
\end{minipage}\end{equation}

Returning to~\eqref{d.x_n2}, note that the expressions on the
right are precisely the reduced expressions~\eqref{d.x_0dots}
for the elements of the subquandle of $Q$ generated
by $x$ and $y;$ so a trajectory is a certain enumeration of
a $\!2\!$-generator subquandle.
Let us prove

\begin{proposition}\label{P.inf}
Let $(x_i)_{i\in\Z}$ be a trajectory in an involutory quandle $Q.$
Then the following conditions are equivalent:\\[.2em]
\textup{(i)} \ \ The subquandle $\{x_i\mid i\in\Z\}$ of $Q$ is free
on the generators $x_0,$ $x_1.$\\[.2em]
\textup{(ii)} \ All $x_i$ are distinct.\\[.2em]
\textup{(iii)} $\{x_i\mid i\in\Z\}$ is infinite.
\end{proposition}

\begin{proof}
The equivalence of~(i) and~(ii) follows
from the last assertion of Theorem~\ref{T.ids}
in view of the enumeration~\eqref{d.x_n2}.
The implication (ii)$\!\implies\!$(iii) is immediate;
to complete the proof, it will
suffice to prove $\neg$(ii)$\implies\!\neg$(iii).
So suppose that for some $i\in\Z$ and $m>0$ we have
\begin{equation}\begin{minipage}[c]{35pc}\label{d.i+m}
$x_i\ =\ x_{i+m}.$
\end{minipage}\end{equation}
For any $j,$~\eqref{d.i+m} implies
$x_i\lhd x_{2i-j} = x_{i+m} \lhd x_{2i-j},$ which
by~\eqref{d.traj.i,j} translates to

\begin{equation}\begin{minipage}[c]{35pc}\label{d.j+2m}
$x_j\ =\ x_{j+2m}\,;$
\end{minipage}\end{equation}
so our trajectory is periodic, proving $\neg$(iii).
\end{proof}

When a trajectory satisfies the equivalent conditions
of Proposition~\ref{P.inf}, we
see that for each $n\geq 1,$ the number of terms whose
normal forms, shown in~\eqref{d.x_n2}, have length (number
of variable-symbols) $\leq n$ is exactly $2n.$
So intuitively, an infinite trajectory ``grows linearly''.
Curiously, this is not true if we allow non-reduced expressions.

\begin{proposition}\label{P.2^n}
Let $(x_i)_{i\in\Z}$ be a trajectory in an involutory quandle.
Then for every positive integer $n,$ the following sets are equal.

The set $A_n$ of elements expressible by
arbitrary $\!\lhd\!$-words of length $\leq n$ in $x_0$ and $x_1.$

The set $B_n$ of elements expressible by
$\!\lhd\!$-words of length exactly $n$ in $x_0$ and $x_1,$
with parentheses clustered on {\em the left}.

The set $C_n=\{x_i\mid -2^{n-1} < i\leq 2^{n-1}\}.$

Thus, if the trajectory $(x_i)_{i\in\Z}$ is infinite, then
for each $n,$ the common value of these sets has cardinality~$2^n.$
\end{proposition}

\begin{proof}
Trivially, $A_1 = B_1 = C_1 = \{x_0,\,x_1\};$ so let $n>1,$
and let us inductively assume the desired result for all lower $n.$

Note that by definition, $B_n=(B_{n-1}\lhd x_0)\cup(B_{n-1}\lhd x_1).$
(Here and below, a formula having a set as one or both
arguments of $\lhd$
denotes the set of outputs obtained using elements of
the indicated input-set(s).)
With the help of~\eqref{d.traj.i,j} we
likewise see that $C_n=(C_{n-1}\lhd x_0)\cup(C_{n-1}\lhd x_1).$
(Indeed, $C_{n-1}\lhd x_0$ consists of all $x_i$ with even $i$ in the
indicated range, and $C_{n-1}\lhd x_1$ of all $x_i$ with
odd $i$ in that range.)
Hence $B_n=C_n$ for all $n.$

By definition, $A_n\supseteq B_n,$ so it will suffice to show
that $A_n\subseteq C_n.$
For $n>1,$ all elements of $A_n$ that do not already
lie in $A_{n-1}$ must be members of sets $A_{n-m}\lhd A_m$ with
$1\leq m<n.$
If $m=1,$ then by the inductive
assumption that $A_{n-1}=C_{n-1},$ we are in exactly
the case of the preceding paragraph, and again get the
elements of $C_n.$
If $2\leq m<n,$ our inductive hypothesis implies that the
elements $x_i$ of our trajectory that lie
in $A_{n-m}$ satisfy $-2^{n-3}<i\leq 2^{n-3}$ (since $m\geq 2),$
while the $x_j$ that lie in $A_m$ satisfy
$-2^{n-2}<j\leq 2^{n-2}$ (since $m\leq n-1),$
whence the subscript $2i-j$ of $x_i\lhd x_j\ =\ x_{2i-j}$
will satisfy $-2^{n-1}<2i-j<2^{n-1},$
so in this case too, that element will lie in $C_n.$

The final assertion is clear, given
Proposition~\ref{P.inf}~(iii)$\implies$(i).
\end{proof}

Using the fact that in groups, trajectories have the form
$(x\,z^i)_{i\in\Z},$ we
can get information about a group $G$ from $\r{Core}(G):$

\begin{lemma}\label{L.nth-pwrs}
Given elements $x$ and $w$ of a group $G,$ and an integer $n,$
one can determine from the structure of $\r{Core}(G)$
whether $x^{-1}w$ is an $\!n\!$-th power in $G.$
Namely, this will hold if and only if there exists a
trajectory $(x_i)_{i\in\Z}$ in $\r{Core}(G)$ with $x_0=x$ and $x_n=w;$
equivalently, if and only if there exists $y\in \r{Core}(G)$
such that $w$ is given by the formula for $x_n$ as
in~\eqref{d.x_n2}.\qed
\end{lemma}

It follows in turn that we can tell from
$\r{Core}(G)$ whether $x^{-1}\,w$ is, say,
a product of squares in $G,$ since this is equivalent
to the existence of a sequence of elements
$x\,{=}\,x_{(0)},\,x_{(1)},\,\dots,\,x_{(n)}\,{=}\,w$ with each
$x_{(n-1)}^{-1}\,x_{(n)}$ a square.
Likewise, we see that the structure
of $\r{Core}(G)$ determines whether a property
such as ``every product of squares is a square'', or
``every product of two distinct squares in $G$ has cube the identity''
holds in the group $G.$
On the other hand, we shall see in Lemma~\ref{L.cm_ncm} that
one cannot always tell from the structure
of $\r{Core}(G)$ whether $G$ is abelian.

\section{Orbits of involutory quandles}\label{S.cpnts}

A very degenerate class of quandles $\r{Core}(G)$ is noted in

\begin{lemma}\label{L.exp2}
If $G$ is a group, then the identity
\begin{equation}\begin{minipage}[c]{35pc}\label{d.triv}
For all $x,y\in G,$ $x\lhd y\ =\ y$
\end{minipage}\end{equation}
holds in $\r{Core}(G)$
if and only if $G$ satisfies the identity $x^2=1,$
i.e., has exponent~$2.$
\end{lemma}

\begin{proof}
Condition~\eqref{d.triv} translates to the group-theoretic
identity $x\,y^{-1} x = y,$ equivalently, $(x\,y^{-1})^2 = 1,$
which clearly holds for all $x,y\in G$ if and only if $x^2=1$
for all $x\in G.$
\end{proof}

If an involutory quandle $T$ satisfies~\eqref{d.triv},
then, of course, every subset of $T$ is a subquandle.
Hence given a homomorphism from an involutory quandle $Q$
to such a $T,$ the inverse image of every subset of $T$ is
a subquandle of $Q.$

Every involutory quandle $Q$ has a universal homomorphic
image satisfying~\eqref{d.triv}, whose elements are
the equivalence classes of elements of $Q$ under
the quandle congruence~$\sim$ generated by relations
\begin{equation}\begin{minipage}[c]{35pc}\label{d.eq_gen}
$x\lhd y~\sim~y$\quad $(x,y\in Q).$
\end{minipage}\end{equation}

With the help of~\eqref{d.inv} it is easy
to show that this congruence has the form
\begin{equation}\begin{minipage}[c]{35pc}\label{d.eq}
$x\sim y\iff (\exists\,z_1,\dots,z_n\in  Q)
\ \ y=z_1\lhd(\dots\lhd (z_n\lhd x)\dots).$
\end{minipage}\end{equation}
These equivalence classes are called the {\em orbits} of $Q.$
(The term is used, more generally, in \cite{aut-q}
for the equivalence classes in not necessarily
involutory quandles determined by the equivalence relation
generated by~\eqref{d.eq_gen}.)
We see that the union of any family of orbits is a subquandle of $Q.$

However, in contrast to the case
described in Lemma~\ref{L.exp2}, the structure of a general quandle $Q$
is not determined by
the separate quandle structures of its orbits: Though
each map $x\lhd -$ $(x\in Q)$ takes every orbit $Q_0$ of $Q$
into itself, if $x$ is not in $Q_0,$
the involution $x\lhd -$ on $Q_0$ carries information not
determined by the $\!\lhd\!$-structure of $Q_0.$

Instead of constructing subquandles $Q'$ of $Q$ by
letting each orbit of $Q$ either wholly belong to $Q'$
or be wholly absent, can we put together a $Q'$ by choosing
subquandles of the various orbits
of $Q$ more or less independently?
Specifically, suppose we start with $\r{Core}(G)$
for $G$ a group, and let $N$ be the normal subgroup of $G$
generated by the squares, so that $G/N$ is
the universal exponent-\!$2$\! image of $G.$
Can we get a subquandle of $\r{Core}(G)$ whose
intersections with the various cosets of $N$ include cosets
of distinct subgroups of $N$?
For instance, can we do this when $G$ is an infinite
cyclic group $\langl x\rangl$?
There, $N=\langl x^2\rangl,$ so the two orbits are the sets of
even and odd integers.

The answer turns out to be no in that case, but yes for
some other $G.$

The negative answer for
$Q=\r{Core}(\langl x\rangl)$ follows from the fact, not hard to
see, that in that quandle, the elements of every nonempty
subquandle $Q'$ form a subtrajectory.
(Idea: If $Q'$ has more than one
element, choose distinct $x^i,\,x^j\in Q'$ so as
to minimize $|i-j|,$ and show that the existence of
an element $x^k$ not in the subtrajectory they generate
would contradict that minimality.)
For such a subtrajectory-determined subquandle $Q',$ the
set $\{i\mid x^i\in Q'\}$ either consists entirely of even integers,
or consists entirely of odd integers,
or the sets of even and of odd elements
are cosets of a common subgroup of $\Z.$

But for an example where more interesting things can happen,
let $G$ be the infinite dihedral group
$\langl x,\,y\mid y^2=1,\ y^{-1}\,x\,y=x^{-1}\rangl.$
It is easy to check that each coset of the subgroup
$\langl x\rangl\subseteq G$ has trivial $\!\lhd\!$-action on the other:
\begin{equation}\begin{minipage}[c]{35pc}\label{d.gi,gjh}
$x^i\,\lhd\,(x^j y)\ =\ x^j y$\quad and \quad
$(x^j y)\,\lhd\,x^i\ =\ x^i$ \quad $(i,j\in\Z).$
\end{minipage}\end{equation}

Hence the union of any subquandle
of one coset with any subquandle
of the other gives a subquandle of $\r{Core}(G);$
and those subquandles
can, independently, each be a nontrivial subtrajectory,
or a singleton, or empty.

(The cosets of $\langl x\rangl$
are not actually the orbits of $\r{Core}(G);$
each is the union of two such orbits.
But each coset of $\langl x\rangl$ is
a trajectory, so, as discussed above,
the intersections of a subquandle of $\r{Core}(G)$ with the two
orbits comprising one of these cosets have much less freedom.)

\section{Which involutory quandles embed in core quandles of groups?}\label{S.embed}

Not every involutory quandle has the form $\r{Core}(G).$
For instance, letting $G$ be a group of exponent~$2,$
we have noted that every subset of $\r{Core}(G)$ is a subquandle.
But if such a subset has
finite cardinality not a power of $2,$ that subquandle clearly
cannot be isomorphic to $\r{Core}(H)$ for any group $H.$

Is every involutory quandle at least {\em embeddable} in one of the
form $\r{Core}(G)$?

No.
A hint of what can go wrong was seen in the proof of
Proposition~\ref{P.inf},
where for a trajectory satisfying a relation $x_i=x_{i+m},$
we did not deduce $x_j=x_{j+m}$ for all $j,$ as is clearly
true in a group-theoretic trajectory, but only $x_j=x_{j+2m}.$
The next result analyzes that behavior in detail;
in Proposition~\ref{P.fix_fix} we note the consequences
for embeddability of involutory quandles in core quandles.

\begin{proposition}\label{P.fin_traj}
Let $Q$ be an involutory quandle,
and $(x_i)_{i\in\Z}$ a trajectory in $Q$ in which not all
terms are distinct.
Then for some positive integer $n,$
\begin{equation}\begin{minipage}[c]{35pc}\label{d.i+n}
$x_i\ =\ x_{i+n}$ for all $i\in\Z.$
\end{minipage}\end{equation}

Let $n$ be the {\em least}
positive integer for which~\eqref{d.i+n} holds.
Then exactly one of the following is true.\\[.2em]
\textup{(i)} \ \ $x_i=x_j$ if and only if
$i\equiv j~(\r{mod}\,n).$\\[.2em]
\textup{(ii)} \ $n$ is a multiple of $4,$ and for $i,j\in\Z$
we have $x_i=x_j$ if and only if either $i$ and $j$ are both
odd, and are congruent modulo $n/2,$ or they are both
even, and are congruent modulo $n.$\\[.2em]
\textup{(ii\!$'$\!)}\, Like \textup{(ii)}, but with ``even'' and
``odd'' interchanged.\vspace{.2em}

Moreover, for each of \textup{(i), (ii), (ii\!$'$\!)},
and all values of $n$ with the indicated properties, there
exist trajectories $(x_i)_{i\in\Z}$ in involutory quandles
$Q$ of the sort described.
\end{proposition}

\begin{proof}
\eqref{d.i+n} holds for some $n$ by the
implication \eqref{d.i+m}$\!\implies\!$\eqref{d.j+2m}
in the proof of Proposition~\ref{P.inf}.
Let $n$ be the least such value.

For each $x\in\{x_i\mid i\in\Z\},$ let $r(x)$ be the least
distance between occurrences of $x$ in our trajectory,
i.e., the least $m>0$ such that for some $i,$ $x_i=x=x_{i+m}.$
Note that if $r(x)\neq n,$ so that $x$ occurs more than once
in a cycle of length $n,$ then we must have $r(x)\leq n/2.$

On the other hand, again calling on the implication
\eqref{d.i+m}$\!\implies\!$\eqref{d.j+2m}, we see that
$x_i=x_{i+2r(x)}$ for all $i,$ so by our choice of $n,$ $n$
is a divisor of $2\,r(x),$ so $n\leq 2\,r(x),$ i.e., $r(x)\geq n/2.$
In view of the conclusion of the
preceding paragraph, this says that if $r(x)\neq n,$ then $r(x)=n/2.$
So for each $x,$ either $r(x)=n,$ in which case $x$
occurs periodically with period $n,$ or $r(x)=n/2,$ so $x$
must occur with period $n/2.$
(Of course, the latter is only possible if $n$ is even.)

Assuming $r(x)=n/2,$ let $x=x_i,$ and let us
apply to the relation $x_i=x_{i+n/2}$ the operator $x_{i+1}\lhd.$
By~\eqref{d.traj.i,j} we get $x_{i+2}=x_{i+2-n/2};$ so we must
also have $r(x_{i+2})=n/2.$
Thus, for $j\in\Z,$ whether $r(x_j)$ is $n$ or $n/2$ can
only depend on the parity of $j.$

We will have established the main assertion of the proposition once
we say why we can't have $r(x_i)=n/2$ for both odd and even $i,$ and
why $n$ must be a multiple of $4$ (and not just an even integer)
in cases~(ii) and~(ii\!$'$\!).
The former point is trivial: if $r(x_i)$ were $n/2$ for both
odd and even $i,$ then for all $i$ we would have $x_i=x_{i+n/2},$
so $n/2,$ not $n,$ would be the least period of $(x_i)_{i\in\Z}.$
To see the other point, note that if $x_i=x_{i+n/2},$ then
$r$ has the value $n/2$ at both $x_i$ and $x_{i+n/2}.$
If $n/2$ were odd, this would mean that both odd- and even-indexed
elements satisfied $r(x)=n/2,$ which we have just noted is
impossible.

It remains to show that all the cases of~(i),~(ii) and~(ii\!$'$\!)
do occur.
For every $n,$ the quandle \mbox{$\r{Core}(\langl x\mid x^n=1\rangl)$}
gives a trajectory~$(x^i)_{i\in\Z}$ as in~(i).
If $n$ is a multiple of $4,$ it is straightforward to verify
that the equivalence relation on the above quandle which
identifies $x^i$ with $x^{i+n/2}$ when and only when $i$
is odd (respectively, even) is a congruence on that
quandle, giving examples of~(ii) and~(ii\!$'$\!) respectively.
(Note that $n$ must be a multiple of $4$ for our description
of this quandle to make sense, i.e., for
$i$ and $i+n/2$ to be of the same parity.)
\end{proof}

Incidentally, note that from a trajectory as
in~(ii) above, one gets a trajectory as in~(ii\!$'$\!)
by shifting the indexing by $1,$ and vice versa; hence
the presence of one sort in a given $Q$
is equivalent to the presence of the other.
So below, we shall only refer to trajectories of the former sort.

\begin{proposition}\label{P.fix_fix}
The following conditions on an involutory quandle $Q$ are
equivalent.\\[.2em]
\textup{(i)}\, $Q$ has no finite trajectories of the
sort described in Proposition~\ref{P.fin_traj}\textup{(ii)}.\\[.2em]
\textup{(ii)} If $(x_i)_{i\in\Z}$ is a trajectory in $Q$
and $i,$ $j,$ $m$ are integers, then
$x_i=x_{i+m}\iff x_j=x_{j+m}.$\\[.2em]
\textup{(i$\!'\!$)}\, $Q$ has no finite trajectories of the
sort described in Proposition~\ref{P.fin_traj}\textup{(ii)}
with $n$ a power of $2.$\\[.2em]
\textup{(ii$\!'\!$)} If $(x_i)_{i\in\Z}$ is a trajectory in $Q$ and
$m\geq 2$ is a power of $2,$ then $x_0=x_m\implies x_1=x_{m+1}.$
\vspace{.2em}

Moreover, every involutory quandle $Q$ that is embeddable in
the core quandle of a group satisfies the above equivalent conditions.
\end{proposition}

\begin{proof}
In the light of Proposition~\ref{P.fin_traj}, it is clear that
(i)$\iff$(ii) and (i$\!'\!$)$\iff$(ii$\!'\!$) (where
the $m$ of (ii$\!'\!$) is half the $n$ of (i$\!'\!$)), and clearly
the former conditions imply the latter conditions.
Conversely, suppose $Q$ fails to satisfy~(i), i.e.,
has a trajectory $(x_i)_{i\in\Z}$ of the
sort described in Proposition~\ref{P.fin_traj}(ii).
Then writing the period $n$ of that description as $m\,n'$
where $m$ is odd and $n'$ is a power of $2$ (which will be $\geq 4),$
we see that $(x_{m\,i})_{i\in\Z}$ will be a trajectory
of period $n'$ of the sort excluded by (i$\!'\!$) above.

Finally, since in a quandle of the form $\r{Core}(G),$
every trajectory $(x_i)_{i\in\Z}$ has the form $(x\,z^i)_{i\in\Z},$
the conditions $x_0=x_m$ and $x_1=x_{m+1}$ both come
down to $z^m=1,$ from which all of (i)-(ii$\!'\!$) are clear.
\end{proof}

\begin{question}\label{Q.fix_fix}
Are the equivalent conditions of Proposition~\ref{P.fix_fix}
sufficient, as well as necessary, for an involutory
quandle $Q$ to be embeddable in the core quandle of a group?
\end{question}

Digressing from the main subject of this paper, we end this section
with some observations on not-necessarily-involutory quandles, and a
question on these, parallel to Question~\ref{Q.fix_fix}.

\begin{lemma}\label{L.traj_in_noninv}
For the remainder of this section, we shall call a sequence
$(x_i)_{i\in\Z}$ of elements of a {\em not necessarily involutory}
quandle $Q$ a trajectory if it satisfies~\eqref{d.traj+1};
equivalently, if it satisfies the analog of~\eqref{d.traj-1} with
$\lhd^{-1}$ in place of $\lhd.$

If $(x_i)_{i\in\Z}$ is a trajectory in a quandle $Q,$ then\\[.2em]
\textup{(a)}  Every index-translate $(x_{i+r})_{i\in\Z}$ of
$(x_i)_{i\in\Z}$ is again a trajectory.\\[.2em]
\textup{(b)} Writing $x_0=x,$ $x_1=y,$ the formulas of~\eqref{d.x_n2}
for $x_i$ with $i\geq 0$ hold, while
those with $i<0$ become true if $\lhd$
is everywhere replaced by $\lhd^{-1}.$\\[.2em]
\textup{(c)}  For all $i,$ $x_{i+2} = y\lhd (x\lhd x_i).$\\[.2em]
\textup{(d)}  For all integers $i,$ $j$ and $n,$ we have
$x_i = x_j \iff x_{i+2n} = x_{j+2n}.$\\[.2em]
\textup{(e)}  If $Q$ is embeddable in $\r{Conj}(G)$ for $G$ a group,
then for all integers $i,$ $j$ and $n,$ we have
$x_i = x_j \iff x_{i+n} = x_{j+n}.$
\end{lemma}

\begin{proof}
(a)~is immediate from the above definition of a trajectory, and~(b)
is easily proved by induction, with the help of the
fact that a common string $y\lhd(x\lhd(y\lhd\dots))$
or $x\lhd^{-1}(y\lhd^{-1}(x\lhd^{-1}\dots))$
with which two successive terms of~\eqref{d.x_n2}
begin acts on $Q$ by an automorphism,
by~\eqref{d.bij} and~\eqref{d.end}.
(c)~is quickly verified by looking separately at
the four cases $i<-1,$ $i=-1,$ $i=0$ and $i>0.$
Since $y\lhd (x\lhd -)$ is an automorphism of $Q,$ (d)~follows from~(c).

In proving~(e), it suffices to establish the case $n=1.$
Moreover, we can assume without loss of generality that $i<j,$
and then, using~(a), assume $i=0.$
Thus, what we must prove is the equivalence, for $j>0,$
of $x_0 = x_j$ with $x_1 = x_{j+1}.$

If $j$ is odd, say $j=2m+1,$ these equations, expressed using
the operations of $G,$ become
$x=(yx)^m y\,(yx)^{-m}$ and $y=(yx)^m y x y^{-1} (yx)^{-m}.$
If, on the other hand, $j=2m,$ they become
$x = (yx)^{m-1} y x y^{-1} (yx)^{-(m-1)}$ and $y=(yx)^m y\,(yx)^{-m}.$
In each case, the equivalence of the two group-theoretic
relations is straightforward:
Each of the first pair of equations reduces
(on bringing all negative-exponent terms to the opposite
side, and cancelling equal end-terms if these
occur) to $x (yx)^m = (yx)^m y;$
each of the second pair to $(xy)^m = (yx)^m.$
\end{proof}

It is not clear to me how natural the concept of trajectory is
in non-involutory quandles.
When $Q=\r{Conj}(G),$ a trajectory in $Q$ will not,
in general, be a trajectory in the group $G$
as defined in~\eqref{d.traj}.
The terms of a trajectory in a quandle $Q$
do not, in general, comprise a subquandle
(e.g., they do not, in general, include $y\lhd(y\lhd x)).$
In particular, trajectories do not,
in general, satisfy~\eqref{d.traj.i,j}.
The failure of that condition means that if $(x_i)_{i\in\Z}$ is a
trajectory,
then for $n\neq 1,$ $(x_{ni})_{i\in\Z}$ will not in general be one.
(Indeed, it will not be one for $n=-1$ if $x\lhd y \neq x\lhd^{-1} y.)$

Nevertheless, point~(e) of the above lemma suggests
the following analog of Question~\ref{Q.fix_fix}.

\begin{question}\label{Q.fix_fix_Conj}
Suppose $Q$ is a \textup{(}not necessarily involutory\textup{)}
quandle such that for all trajectories $(x_i)_{i\in\Z}$ in $Q$
and all positive integers $j,$ we have
\begin{equation}\begin{minipage}[c]{35pc}\label{d.fix_fix_noninv}
$x_0 = x_j\iff x_1 = x_{j+1}.$
\end{minipage}\end{equation}
Must $Q$ be embeddable in $\r{Conj}(G)$ for some group $G$?
\end{question}

We remark that for
trajectories in quandles obtained from groups $G$ by the
formula~\eqref{d.conj_d} with $|d|>1,$~\eqref{d.fix_fix_noninv}
need {\em not} hold.
For example, one finds that in such a quandle, the left-hand equation of
the $j=2$ case of~\eqref{d.fix_fix_noninv} says
that in $G,$ $y^d$ commutes with $x,$ and the right-hand equation says
that $x^d$ commutes with $y;$
but if we take for $G$ a group having elements $x$ and $y$ which do
not commute, and such that $x$ has order prime to $d,$ while
$y$ has order dividing $d,$ then for these $x$ and $y,$ the first of
the above conditions clearly holds, while the second fails.

I do not know whether there are interesting conditions on quandles
that {\em are} implied by embeddability in quandles so obtained
for values of $d>1.$
(The conditions satisfied for $d$ and for $-d$ can be obtained from
each other by interchanging $\lhd$ and $\lhd^{-1},$ so the
cases with negative $d$ don't have to be examined separately.)

\section{More on mapping involutory quandles into core quandles of groups}\label{S.to_gp}

Proposition~\ref{P.fix_fix} gives us restrictions on
involutory quandles embeddable in quandles of the form $\r{Core}(G).$
Nevertheless, there is a natural homomorphism of any involutory
quandle into a core quandle,
which often does a good job of separating elements.

\begin{proposition}\label{P.rep}
Let $Q$ be any involutory quandle, and $\r{Perm}(Q)$ the group
of all permutations of the set $Q.$
For $x\in Q,$ define $\overline{x}\in\r{Perm}(Q)$ by
\begin{equation}\begin{minipage}[c]{35pc}\label{d.rep}
$\overline{x}(a)\ =\ x\lhd a$\quad $(a\in Q).$
\end{minipage}\end{equation}
Then $x\mapsto\overline{x},$ is a quandle homomorphism
$Q\to\r{Core}(\r{Perm}(Q)).$

If $Q$ above has the form $\r{Core}(G)$ for a group $G,$ then
elements $x,\,x'\in Q$ fall together under this homomorphism
if and only if, as members of $G,$ they belong to the same coset
of the group of elements of exponent $2$ in the {\em center} of~$G.$
\end{proposition}

\begin{proof}
That the maps $\overline{x}: Q\to Q$ are invertible,
i.e., belong to $\r{Perm}(Q),$
is property~\eqref{d.bij}.
(Though we did not make~\eqref{d.bij}
part of our characterization of involutory quandle
in Lemma~\ref{L.redef}, we
noted that it follows from~\eqref{d.inv}, which says
that every map $\overline{x}$ has exponent~$2.)$

To check that $x\mapsto\overline{x}$ is a homomorphism of quandles,
let $x,y\in Q.$
Then we see (using~\eqref{d.end2} at
the second step below, and the fact that $\overline{y}$ has
exponent $2$ at the fourth) that for all $z\in Q,$
\begin{equation}\begin{minipage}[c]{35pc}\label{d.rep.hm}
$(\overline{x\lhd y})(z)\ =
\ (x\lhd y)\lhd z\ =
\ x\lhd(y\lhd(x\lhd z))\ = \\[.2em]
\hspace*{2em}(\overline{x}\ \overline{y}\ \overline{x})(z)\ =
\ (\overline{x}\,\overline{y}^{-1}\overline{x})(z)\ =
\ (\overline{x}\lhd\overline{y})(z),$
\end{minipage}\end{equation}
so $\overline{x\lhd y} = \overline{x}\lhd\overline{y},$ as required.

To get the last assertion of the proposition, note that for
$x,\,x'\in G,$ we have $\overline{x}=\overline{x'}$
if and only if all $y\in G$ satisfy
$x\,y^{-1} x = x' y^{-1} x'.$
Multiplying on the left by $x'^{-1}$ and on the right by $x^{-1},$
this becomes
\begin{equation}\begin{minipage}[c]{35pc}\label{d.xx'y}
$x'^{-1} x\,y^{-1}\ =\ y^{-1} x'\,x^{-1}.$
\end{minipage}\end{equation}

Taking $y=1$ in~\eqref{d.xx'y} gives
\begin{equation}\begin{minipage}[c]{35pc}\label{d.xx'}
$x'^{-1} x\ =\ x' x^{-1}.$
\end{minipage}\end{equation}
Hence~\eqref{d.xx'y} says that the common
value of the two sides of~\eqref{d.xx'} is central in $G.$
Hence, in particular, the right-hand side of~\eqref{d.xx'}
is unaffected by conjugation by $x;$
but the result of that conjugation is the inverse of the left-hand
side, so the common value of the two sides also has exponent $2,$
giving the ``only if'' direction of the desired statement.
The ``if'' direction is straightforward.
\end{proof}

(The map $\overline{x}$ of~\eqref{d.rep}, for $Q$ a
not necessarily involutory quandle, is called
in~\cite[Definition~1.1]{qndl} the {\em symmetry} $S(x)$ of $Q$ at $x;$
it is an automorphism of $Q.)$
% the subgroup of $\r{Aut}(Q)$ generated by these symmetries is there
% named $\r{Inn}(Q),$ the group of {\em inner} automorphisms of $Q.$

If elements of an involutory quandle $Q$ fall together under the
map $Q\to\r{Core}(\r{Perm}(Q))$ of Proposition~\ref{P.rep},
this may be because $Q$ cannot be embedded in
the core quandle of a group, as is the case for
the trajectories of Proposition~\ref{P.fin_traj}\textup{(ii)};
or that may not be so, as we see from the
last paragraph of Proposition~\ref{P.rep}.

To avoid ``unnecessary falling-together'',
one can try to embed $Q$ in a larger
involutory quandle $Q',$ such that even if
two elements $x\neq x'$ satisfy $x\lhd y = x'\lhd y$ for all $y\in Q,$
this equality fails for some $y\in Q',$
so that Proposition~\ref{P.rep} yields
a representation of $Q'$ that distinguishes them.
When $Q$ has the form $\r{Core}(G)$ for some group $G,$
this will always work: construct a group $H$ by adjoining to $G$
one new generator $z$ and no relations.
Then nonidentity elements of $G$ will not centralize $z,$ so in
$\r{Perm}(\r{Core}(H))$ the cases of elements of $G$ falling
together as described in Proposition~\ref{P.rep} become trivial.

Given an arbitrary involutory quandle $Q,$ there will similarly
exist a universal involutory quandle $Q'$ generated by an image
of $Q$ and one additional generator $z.$
If we could find a normal form for elements of this $Q'$
in terms of $Q,$ we could use it to tell
which pairs of elements $x,\,x'$ fall together under
all maps into core quandles of groups.
(Namely, if and only if $x\lhd z=x'\lhd z.)$
But I do not see how to get such a normal form.
Obviously, we can reduce any element of $Q'$ to an
expression~\eqref{d.x_0dots} in elements of $Q\cup\{z\}.$
But the identities of involutory quandles will imply further
equalities among such expressions.
For instance, suppose we have an expression
$\ldots \lhd(x_i\lhd(x_{i+1}\lhd(\dots)))\dots$
with $x_i$ and $x_{i+1}$ both coming from $Q.$
Let us use~\eqref{d.inv} in reverse, to insert
two terms $x_i$ after $x_{i+1},$ getting an expression
$\dots\lhd(x_i\lhd(x_{i+1}\lhd(x_i\lhd(x_i\lhd(\dots)))))\dots,$
then apply~\eqref{d.end2} to the first three of the terms shown.
Then our element becomes
$\dots \lhd(x'_{i+1}\lhd(x_i\lhd(\dots)))\dots,$
where $x'_{i+1}=x_i\lhd x_{i+1}\in Q.$
For another example: if five successive terms $x_i,\dots,x_{i+4}$
all come from $Q$ and satisfy
$x_i=x_{i+2}=x_{i+4},$ then we can apply~\eqref{d.end2} either
to $x_i,\,x_{i+1},\,x_{i+2},$ or to $x_{i+2},\,x_{i+3},\,x_{i+4},$
getting different reductions of our expression.

Contrast this with the case of the {\em group} gotten by adjoining a new
generator $z$ to an arbitrary group $G.$
This has a normal form consisting of all
alternating strings of nonidentity elements of $G$ and
nonzero powers of $z,$
from which one quickly sees that no nonidentity
element of $G$ is central in the new group.

The reader might find it interesting to examine the case
where $Q$ is the involutory quandle
of Proposition~\ref{P.fin_traj}(ii) with $n=4,$
consisting of the three elements $x_0,$ $x_1=x_3,$ and $x_2,$
and see how the axioms for
an involutory quandle force $x_0\lhd z = x_2\lhd z.$
(Outline: In $x_2\lhd z,$ substitute
$x_1\lhd x_0$ for $x_2,$ and expand the result using~\eqref{d.end2}.
Write the last $x_1$ in the resulting expression
as $x_0\lhd x_1$ and again expand by~\eqref{d.end2}.
Then apply~\eqref{d.inv} twice.
This $\!3\!$-element quandle is called $\r{Cs}(4)$
in~\cite[next-to-last paragraph of {\S}6]{qndl}.)

By general nonsense
(see \cite[Exercise~9.9:8, or better, Theorem~10.4:3]{245})
one can associate to any involutory quandle
$Q$ a group $\r{Group}(Q)$ with a universal
involutory quandle homomorphism $Q\to \r{Core}(\r{Group}(Q)).$
The pairs $x,\,x'$ of elements of $Q$
that fall together under this homomorphism will be
those that fall together under all homomorphisms to core
quandles of groups.
But, as in the approach of adjoining a universal $z$
to $Q$ as an involutory quandle, it is not clear how to get
a good enough picture of $\r{Group}(Q)$ to detect such pairs.

Incidentally, the abovementioned universal homomorphism
$Q\to\r{Core}(\r{Group})(Q)$ can never be surjective.
To see this, take any nontrivial group $G$ and any $g\in G-\{1\}.$
Then a homomorphism $c_g: Q\to \r{Core}(G)$ is given by
the constant map $c_g(x)=g$ $(x\in Q).$
By the universal property of $\r{Group}(Q),$ $c_g$ must
factor $Q\to \r{Core}(\r{Group}(Q))\to \r{Core}(G),$ where the second
map is induced by some group homomorphism $\r{Group}(Q)\to G.$
Since $c_g$ takes no element of $Q$ to $1\in G,$
our map $Q\to\r{Core}(\r{Group}(Q))$ cannot take any element
of $Q$ to $1\in\r{Core}(\r{Group}(Q)),$ and so cannot be surjective.

Returning to our observation that every involutory quandle $Q$ that
can be embedded in the core quandle of a group $G$ can in fact be
embedded in $\r{Core}(\r{Perm}(H))$ for an appropriate
overgroup $H$ of $G,$ by sending each $x\in G$ to
the permutation~\eqref{d.rep} of the underlying set of $H,$
note that each of these permutations~\eqref{d.rep} has exponent~$2.$
We record this, along with some straightforward observations, in
the next result (where $\r{Inv}$ stands for ``set of involutions'').

\begin{proposition}\label{P.exp_2}
For any group $G,$ the elements of exponent $2$ in $G$ form
a subquandle $\r{Inv}(G)$ both of $\r{Core}(G)$ and of $\r{Conj}(G),$
on which the restrictions of the operations of those two quandles agree.

An involutory quandle $Q$ can be embedded in the core
quandle of a group $G$ if and only if it can be embedded in
the involutory quandle $\r{Inv}(H)$ for some group $H.$ \qed
\end{proposition}

(Cf.~\cite[Corollary~10.3]{qndl}, which shows that the
free involutory quandle on an $\!A\!$-tuple of elements embeds
naturally in $\r{Conj}(G),$ for $G$ the group presented
by an $\!A\!$-tuple of elements of exponent~$2.$
From this the second
assertion of our Theorem~\ref{T.ids} can be deduced.
More generally,~\cite{qndl} defines an {\em $\!n\!$-quandle}
to be a quandle in which the $\!n\!$-th power of each
derived operation $x\lhd$ is the identity, with
the special case $n=2$ being named an involutory quandle.
The cited corollary gives the above result
for $\!n\!$-quandles, and $G$ presented by an $\!A\!$-tuple
of generators of exponent~$n.)$

Proposition~\ref{P.exp_2} shows that if an involutory quandle can be
embedded in a quandle $\r{Core}(G),$ it can also
be embedded in a quandle $\r{Conj}(H).$
I don't know whether the converse is true:

\begin{question}\label{Q.Inv}
Can every {\em involutory} quandle $Q$
that is embeddable in $\r{Conj}(G)$
for some group $G$ be embedded in $\r{Inv}(H)$ for some group $H$?
Equivalently \textup{(}by the final statement
of Proposition~\ref{P.exp_2}\textup{)}
can any such $Q$ be embedded in the core quandle of a group?
\end{question}

An involutory subquandle of $\r{Conj}(G)$ that
is, in general, larger than $\r{Inv}(G)$ is the set of
elements whose squares are {\em central}.
This set also forms a subquandle of $\r{Core}(G),$ but the two
quandle structures are, in general, distinct (easily seen
in the case $G=Z).$

We remark that for $G$ a group, we can get an embedding
of $\r{Core}(G)$ in a
quandle of the form $\r{Inv}(H)$ as in the second statement of
Proposition~\ref{P.exp_2} using a
group $H$ that is not as enormous as
the permutation group of the underlying set of the group gotten by
freely adjoining a new element $z$ to $G,$
discussed in the third paragraph following Proposition~\ref{P.rep}.
To motivate the description of such an $H,$
let $G^*$ be the result of adjoining such a $z$ to $G.$
Note that the $\!\lhd\!$-actions on $G^*$ of elements of $G$
carry into itself the subset $G\,\{z,\,z^{-1}\}\,G,$ and
that their actions on that subset belong to the subgroup of
$\r{Perm}(G\,\{z,\,z^{-1}\}\,G)$ generated by left multiplication by
members of $G,$ right multiplication by
members of $G,$ and the operation~$(\ )^{-1}.$
That subgroup is isomorphic to the semidirect
product $Z_2 \ltimes (G\times G),$ where the nonidentity element
of $Z_2,$ which we shall
denote $u,$ acts on $G\times G$ by interchanging the factors.
Namely, we let elements of the form $(1,x,y)\in Z_2 \ltimes (G\times G)$
act on $G\,\{z,\,z^{-1}\}\,G$ by
$h\mapsto x\,h\,y^{-1},$ and let $u$ act by $h\mapsto h^{-1}.$
Thus, for $x\in G,$ $x\lhd\!-$ is represented by
$(u,x,x^{-1})\in Z_2 \ltimes (G\times G).$
It is straightforward to verify (without calling on
the above motivation) that, in the notation
of Proposition~\ref{P.exp_2},
\begin{equation}\begin{minipage}[c]{35pc}\label{d.G->H}
$x\ \mapsto\ (u,x,x^{-1})$ \quad is an embedding of
involutory quandles \quad
$\r{Core}(G)\to \r{Inv}(Z_2 \ltimes (G\times G)).$
\end{minipage}\end{equation}

In another direction,
let us note a property of quandles of the form $\r{Core}(G)$
which follows immediately from Lemma~\ref{L.auts}(i).

\begin{lemma}\label{L.trans}
If $G$ is a group, then the automorphism group
of $\r{Core}(G)$ is transitive on the underlying set
of that quandle.\qed
\end{lemma}

An easy example of an involutory quandle whose automorphism group
is not transitive is the $\!3\!$-element quandle of
Proposition~\ref{P.fin_traj}(ii) with $n=4$
(mentioned in the fourth paragraph before
Proposition~\ref{P.exp_2} above).
The element $x_1=x_3$ is fixed under all the operations $x\lhd -,$
while the elements $x_0$ and $x_2$ are not; so no automorphism
of $Q$ can carry $x_1=x_3$ to $x_0$ or $x_2.$

The above example is a homomorphic image of a quandle
of the form $\r{Core}(G).$
A quandle $Q$ which is, rather, embeddable in one
of the form $\r{Core}(G),$ but again
does not have transitive automorphism group, is the
case of the example in the paragraph containing~\eqref{d.gi,gjh}
where, as subsets of the cosets $\langl x\rangl$ and
$\langl x\rangl y,$ we take all of one coset,
and a singleton subset of the other.
Then $Q$ consists of an infinite trajectory
together with a lone element which
belongs only to trajectories of $\leq 2$ elements;
so no automorphism can carry that element to any other.

\section{Identities in groups and their core quandles}\label{S.ids}

If a group $G$ satisfies nontrivial identities
(identities not implied by the
identities defining groups), this can lead to
nontrivial identities on the involutory quandle $\r{Core}(G).$
We saw this in an extreme way in Lemma~\ref{L.exp2};
let us examine some other examples.

\begin{theorem}\label{T.ids_ab}
The identities satisfied by $\r{Core}(G)$ for all
{\em abelian} groups $G$ are the consequences \textup{(}given the
defining identities~\eqref{d.idpt},~\eqref{d.inv}
and~\eqref{d.end2}\textup{)} of the identity
\begin{equation}\begin{minipage}[c]{35pc}\label{d.id_ab}
$w\lhd(x\lhd(y\lhd z))\ =\ y\lhd(x\lhd(w\lhd z)).$
\end{minipage}\end{equation}
\end{theorem}

\begin{proof}
That~\eqref{d.id_ab} holds in $\r{Core}(G)$ when $G$ is
abelian is immediate.
(Cf.\ the group-theoretic expansion~\eqref{d.x_0dots_gp}
of the general involutory quandle expression ~\eqref{d.x_0dots}.)

To see that the only identities holding in all
such quandles are the consequences of~\eqref{d.id_ab},
first note that given an expression
$x_0\lhd(x_1\lhd( \dots \lhd(x_{n-1}\lhd x_n)\dots)),$
where the $x_i$ are symbols in a set $X,$ we can,
using~\eqref{d.id_ab}, rearrange in any way the $x_i$
having {\em even} subscripts $i<n,$ and likewise
rearrange in any way the $x_i$ having {\em odd} subscripts $i<n.$
In particular, if some $x\in X$ occurs in both even and
odd positions, we can rearrange the terms so that these
occurrences of $x$ appear in adjacent positions,
and then use~\eqref{d.idpt} or~\eqref{d.inv} to shorten the word.
(We use~\eqref{d.idpt} if one of these occurrences of
$x$ is $x_n,$ so that it
was the other occurrence that had to be moved to become
adjacent to it;~\eqref{d.inv} if neither occurrence is $x_n,$
so that one, the other, or both could be moved to make them adjacent.)

Now suppose that $u=v$ is an identity in symbols from $X$ satisfied by
$\r{Core}(G)$ for all abelian groups $G.$
Using~\eqref{d.end2} we can assume without loss of generality
that in both $u$ and $v,$ parentheses are clustered to the right, while
using~\eqref{d.id_ab}, \eqref{d.idpt}, and~\eqref{d.inv} as
above, we can assume that in each of these words, no member
of $X$ occurs in both even-subscripted and odd-subscripted positions.
Let us now evaluate $u$ and $v$ in the free abelian group $G$ on $X$
(which we will write multiplicatively).
For $x\in X,$ an occurrence of $x$ as the $\!i\!$-indexed term
of the expression $u$ or $v$ will contribute $(-1)^i\,2$ to
the exponent of $x$ in the resulting element of $G,$
unless the term in question is the final
term (i.e., $x_n$ if our expression
is $x_0\lhd(\dots \lhd x_n)\dots\,)$
in which case it will contribute just $(-1)^{i}.$
Since no term occurs in both even and odd positions
in $u$ or in $v,$ we can conclude from
the structure of free abelian groups that $u$ and $v$ must have
the same length, the same number of
occurrences of each element of $X$ in nonfinal even position,
the same number of
occurrences of each element of $X$ in nonfinal odd position,
and the same final term.
Hence $u$ can be transformed into $v$ by
applications of~\eqref{d.id_ab}; hence the identity $u=v$ is indeed a
consequence of \eqref{d.idpt}, \eqref{d.inv}, \eqref{d.end2},
and~\eqref{d.id_ab}.
\end{proof}

What about the other direction?
I.e., for which groups $G$ will $\r{Core}(G)$ satisfy~\eqref{d.id_ab}?

\begin{theorem}\label{T.nilp}
If $G$ is a group, then $\r{Core}(G)$ satisfies~\eqref{d.id_ab} if
and only if $G$ is nilpotent of nilpotency class $\leq 2,$
i.e., if and only if
\begin{equation}\begin{minipage}[c]{35pc}\label{d.nilp}
every commutator $[x,y]\ =\ x^{-1} y^{-1} x\,y$ $(x,y\in G)$ is
central in $G.$
\end{minipage}\end{equation}
\end{theorem}

\begin{proof}
The quandle identity~\eqref{d.id_ab} on $\r{Core}(G)$
translates to the group-theoretic identity on $G,$
\begin{equation}\begin{minipage}[c]{35pc}\label{d.id_ab2}
$w\,x^{-1}\,y\,z^{-1}\,y\,x^{-1}\,w\ =
\ y\,x^{-1}\,w\,z^{-1}\,w\,x^{-1}\,y.$
\end{minipage}\end{equation}

Let us start with the case where $x=1,$
write $z^{-1}=u,$ and multiply the
resulting equation both on the left and on the right
by $y^{-1}\,w^{-1}.$
Then we get
\begin{equation}\begin{minipage}[c]{35pc}\label{d.id_ab3}
$u\,y\,w\,y^{-1} w^{-1}\ =\ y^{-1} w^{-1} y\,w\,u.$
\end{minipage}\end{equation}
Taking $u=1,$ this tells us that $y\,w\,y^{-1}\,w^{-1}$ equals
$y^{-1}\,w^{-1}\,y\,w,$ i.e., $[y,\,w].$
So the general case of~\eqref{d.id_ab3} tells us that
$u\,[y,\,w]=[y,\,w]\,u;$
so indeed, every commutator $[y,\,w]$ in $G$ is central.

Conversely, suppose that in $G$ every commutator is central.
Note that in any group, an element $w\,x^{-1}\,y$ (such as
we have at the beginning of the left side of~\eqref{d.id_ab2} and the
end of the right side) differs from $y\,x^{-1}\,w$
(as at the end of the left side of~\eqref{d.id_ab2} and the
beginning of the right side) by a product of commutators.
Since commutators are central, if we multiply
each side of~\eqref{d.id_ab2} by that product of
commutators, we can let that product act on the beginning of
the left side and the end of the right side,
reducing~\eqref{d.id_ab2} to the trivial identity
$(y\,x^{-1}\,w)\,z^{-1}\,(y\,x^{-1}\,w)
=(y\,x^{-1}\,w)\,z^{-1}\,(y\,x^{-1}\,w).$
So~\eqref{d.id_ab2} indeed holds in every group where
commutators are central.
\end{proof}

If a quandle of the form $\r{Core}(G)$ satisfies~\eqref{d.id_ab},
can it also be written $\r{Core}(H)$ for an abelian group $H$?
In general, no, as seen in the final statement of

\begin{lemma}\label{L.sqs}
For any group $G,$ the following conditions are equivalent:

\textup{(i)}\ \ In $G,$ every product of squares is a square.

\textup{(ii)}\,\ In $\r{Core}(G),$ for all elements
$x,\,y,\,z$ there exists an element $w$ such that
\begin{equation}\begin{minipage}[c]{35pc}\label{d.xyzw}
$x\lhd(y\lhd z)\ =\ w\lhd z.$
\end{minipage}\end{equation}

Thus, every abelian group, since it satisfies~\textup{(i)},
satisfies~\textup{(ii)}.
On the other hand, the group $G$ free on two
generators in the variety determined by~\eqref{d.nilp},
equivalently, the group of upper
triangular $3\times 3$ matrices over $\Z$ with $\!1\!$'s
on the diagonal, does not satisfy~\textup{(ii)}.
Hence the core quandle of the latter group, though it
satisfies~\eqref{d.id_ab}, is not isomorphic
to the core quandle of an abelian group.
\end{lemma}

\begin{proof}
In the paragraph following Lemma~\ref{L.nth-pwrs}, we noted
that condition~(i) above could be expressed in terms of the structure
of $\r{Core}(G);$ condition~(ii) is the explicit form that
that condition takes.
(Idea: multiplying a group element on one side by
an $\!n\!$-th power corresponds to moving it $n$ steps
along some trajectory; and by~\eqref{d.x_n2}, moving
an element $x_0$ two steps along a trajectory
is equivalent to applying some operation $y\lhd$ to it.
So~\eqref{d.xyzw} says that the result of multiplying $z\in G$
on a given side by a square, and then by another square,
can always be achieved by multiplying it by a single square.)

That every abelian group satisfies~(i) is clear.
To see that the free group $H$ of nilpotency class $2$ on
generators $x,\,y$ does not satisfy~(ii), let us
write the general element thereof as
$x^i\,y^j\,[y,x]^k$ $(i,\,j,\,k\in\Z),$
and note that the group operation is given by
\begin{equation}\begin{minipage}[c]{35pc}\label{d.nilpxy}
$(x^i\,y^j\,[y,x]^k)\ (x^{i'}\,y^{j'}\,[y,x]^{k'})\ =
\ x^{i+i'}\,y^{j+j'}\,[y,x]^{k+k'+ji'}.$
\end{minipage}\end{equation}
(Rough idea: $yx = xy\,[y,x],$ so in bringing the product on the
left-hand side to
normal form, each time we push one of the $j$ occurrences of $y$
in the first factor past one of the $i'$ occurrences of $x$
in the second, a $[y,x]$ is created.)

Note that if for this $H,$ we evaluate the left-hand side
of~\eqref{d.xyzw} with the element $y^{-1}$ in the role of $y,$ and
$1$ in the role of $z,$ we get, in group-theoretic terms,
$x\,y\,1\,y\,x,$ i.e., $x\,y^2\,x,$ which by~\eqref{d.nilpxy} equals
\begin{equation}\begin{minipage}[c]{35pc}\label{d.xy1yx}
$x^2\,y^2\,[y,x]^2.$
\end{minipage}\end{equation}

Now suppose that some choice of $w=x^i\,y^j\,[y,x]^k$ makes
the right-hand side of~\eqref{d.xyzw} equal~\eqref{d.xy1yx}.
Since we have taken $z=1,$ that right-hand side
is $w^2,$ and we see
from~\eqref{d.nilpxy} that for this to equal~\eqref{d.xy1yx},
we must have $i=j=1.$
But this gives the exponent of $[y,x]$ the value $2k+1,$
so the expression cannot agree with~\eqref{d.xy1yx}.
\end{proof}

On the other hand, if we adjoin to the nilpotent group of the above
lemma a central square root of $[y,x],$ the above problem goes away:

\begin{lemma}\label{L.cm_ncm}
Let $G$ be the free abelian group on three generators $x,\,y,\,z,$
and $H$ the group obtained by adjoining to the free
group of nilpotency class $2$ on
generators $x,\,y$ a central square root of the
element $[y,x],$ which we shall write $[y,x]^{1/2}.$
Thus, the general element of $H$ can be written in the
normal form $x^i\,y^j\,[y,x]^{k/2}$ with $i,\,j,\,k\in\Z,$
and the group operation of $H$ is given by~\eqref{d.nilpxy}
with $k$ and $k'$ everywhere replaced by $k/2$ and $k'/2.$

Then $\r{Core}(H)\cong \r{Core}(G),$ by the map
\begin{equation}\begin{minipage}[c]{35pc}\label{d.k/2}
$x^i\,y^j\,[y,x]^{k/2}\ \mapsto\ x^i\,y^j\,z^{k-ij}.$
\end{minipage}\end{equation}
\end{lemma}

\begin{proof}
The map~\eqref{d.k/2} is clearly a bijection.
Computation shows that it respects $\lhd.$
\end{proof}

(The computation of the exponent of $z$ in the image of
the $\!\lhd\!$-product of two elements of $H$ is messy; I wish
I could offer a nicer verification.)
% For a bit of intuition on~\eqref{d.k/2}, observe that
% $x^i\,y^j\,[y,x]^{k/2}$ can also be written
% $y^j\,x^i\,[y,x]^{(k/2)-ij},$
% and that $(k/2)\,+\,((k/2)-ij)=k-ij$ balances the
% asymmetry implicit in each of these two normal forms for $H.$

Returning to the consequence of Lemma~\ref{L.sqs}, that the core
quandles of the free abelian group of rank three and the
free group of nilpotency class~$2$ on two generators
are not isomorphic, we remark that each can nonetheless
be embedded in the other.
In one direction, restricting~\eqref{d.k/2} to the
case where $[y,x]$ has integer exponent, we get a $\!\lhd\!$-embedding
of the free nilpotent group in the free abelian group,
\begin{equation}\begin{minipage}[c]{35pc}\label{d.np->ab}
$x^i\,y^j\,[y,x]^k\ \mapsto\ x^i\,y^j\,z^{2k-ij}.$
\end{minipage}\end{equation}

For the other direction, note that the inverse
of~\eqref{d.k/2} carries the subgroup of $G$ generated
by $x,$ $y^2$ and $z^2,$ which is free abelian on these generators,
into the subgroup of $H$ consisting of elements in which $[y,x]$
has integer exponent, i.e., our free group of nilpotency class~$2.$

Turning back to the identity~\eqref{d.id_ab},
here is another way to look at that condition.

\begin{lemma}\label{L.ab_in_Perm}
Let $Q$ be a nonempty involutory quandle, and let us fix an
arbitrary element $u\in Q.$
Then $Q$ satisfies~\eqref{d.id_ab} if and only if
\textup{(}in the notation of Proposition~\ref{P.rep}\textup{)}
the elements of the set
$\{\,\overline{x}\ \overline{u}\mid x\in Q\}\subseteq \r{Perm}(Q)$ all
commute with one another; in other words, if and only if the map
\begin{equation}\begin{minipage}[c]{35pc}\label{d.*af*au}
$x\ \mapsto\ \overline{x}\ \overline{u},$
\end{minipage}\end{equation}
which is a $\!\lhd\!$-homomorphism $Q\to\r{Core}(\r{Perm}(Q))$
\textup{(}since it is a group-theoretic right translate
of the \mbox{$\!\lhd\!$-homomorphism} $x\mapsto\overline{x}$ of
Proposition~\ref{P.rep}\textup{)},
has image in an abelian subgroup of $\r{Perm}(Q).$

Hence if that condition holds, and if, moreover, the
map $x\mapsto\overline{x}$ is one-to-one, then $Q$
is embeddable in $\r{Core}(G)$ for an abelian group $G.$

In particular, for every group $H$ of nilpotency class $\leq 2$
whose center has no elements of order $2,$
$\r{Core}(H)$ is embeddable in $\r{Core}(G)$ for an abelian group $G.$
\end{lemma}

\begin{proof}
Suppose first that for some $u\in Q,$ the elements
$\overline{x}\,\overline{u}$ $(x\in Q)$ all lie in an abelian
subgroup of $\r{Perm}(Q).$
Since $\overline{u}^2=1,$ these elements can be written
$\overline{x}\,\overline{u}^{-1},$ hence for any $x,\,y\in Q,$
that abelian subgroup contains $(\overline{x}\,\overline{u}^{-1})\,
(\overline{y}\,\overline{u}^{-1})^{-1}=\overline{x}\,\overline{y}^{-1};$
so our hypothesis is equivalent to the statement
(independent of the choice of an element $u)$ that
all elements of $\r{Perm}(Q)$ of the form
$\overline{x}\,\overline{y}^{-1}$ $(x,y\in Q)$ commute.
Again using the fact that the exponent $^{-1}$
on an element of the form $\overline{x}$ makes no difference,
we see in particular that for all $w,\,x,\,y\in Q,$ we have
$(\overline{w}\,\overline{x})\,(\overline{y}\,\overline{x})=
(\overline{y}\,\overline{x})\,(\overline{w}\,\overline{x}),$ which,
cancelling the $\overline{x}$'s on the right, gives
$\overline{w}\,\overline{x}\,\overline{y}=
\overline{y}\,\overline{x}\,\overline{w}.$
Applying this element of $\r{Perm}(Q)$
to elements $z\in Q,$ we get~\eqref{d.id_ab}.

The reverse implication works essentially the same way.

The assertion of the second paragraph follows immediately.
The final assertion then follows in view of the last
assertion of Proposition~\ref{P.rep}.
\end{proof}

(We remark that~\cite[Definition~1.3]{qndl} defines an
``abelian quandle'' to be a not necessarily involutory
quandle satisfying an identity equivalent to~\eqref{d.id_ab},
and that in~\cite[Theorem~10.5]{qndl}, a description is
given of the free abelian involutory quandle on finitely
many generators.)
% In alternative description in last sentence of \cite[\S10]{qndl},
% ``exactly one $k_i$ is odd'' should be supplemented with
% ``$k+0+\dots+k_n=1$''.

Let us take a brief look at the other very simple
sort of identity a group can satisfy, saying that its
elements all have exponent~$n,$ for some fixed~$n.$
The technique of
Lemma~\ref{L.nth-pwrs} shows us that for each~$n,$
the groups satisfying this identity can be characterized
by a $\!\lhd\!$-identity on their core quandles.
Namely, comparing~\eqref{d.x_n} and~\eqref{d.x_n2}, we see

\begin{lemma}\label{L.exp_n}
Let $n$ be a positive integer.
Then a group $G$ satisfies the identity $x^n=1$ if and
only if $\r{Core}(G)$ satisfies the identity equating
the formulas for~$x_0$ and $x_n$ in~\eqref{d.x_n2};
i.e., for $n$ even, the identity $x=y\lhd(x\lhd(\dots\lhd x))$
with $n/2$ $y$\!'s and $n/2$ $x$\!'s in the right-hand expression;
for $n$ odd, $x=y\lhd(x\lhd (\dots\lhd y))$
with $(n+1)/2$ $y$\!'s and $(n-1)/2$ $x$\!'s.\qed
\end{lemma}

The above ``if and only if'' shows that
in this case, we don't have the complication that we had
for commutativity, where the effect of
our $\!\lhd\!$-identity was weaker than
the group identity we started with.
But we have the opposite sort of complication.
For each positive integer $n$ we can ask

\begin{question}\label{Q.exp_n}
Does the $\!\lhd\!$-identity described in Lemma~\ref{L.exp_n}
imply, for general involutory quandles, {\em all}
identities satisfied by the core-quandles of groups of exponent $n$?

Equivalently, is the free involutory quandle $Q$ on any
set of generators, subject to that identity, embeddable in the
involutory quandle of a group of exponent $n$?
\end{question}

Observe that in the quandle $Q$ of the final sentence of
the above question, all trajectories have period dividing $n.$
Hence if $n$ is not divisible by $4,$
so that $Q$ satisfies condition~(i) of Proposition~\ref{P.fix_fix},
an affirmative answer to Question~\ref{Q.fix_fix} would imply
that $Q$ is embeddable in the core quandle of some group.
So assuming the answer to Question~\ref{Q.fix_fix} is affirmative,
suppose $Q\subseteq\r{Core}(G).$
Then the $\!\lhd\!$-identity in question implies that for all
$x,y\in Q,$ the element $xy^{-1}\in G$ has exponent $n.$
By a translation, we may assume that $Q$ contains $1\in G.$
The fact that elements $xy^{-1}$ $(x,y\in Q)$ have exponent
$n$ in $G$ then implies that elements of $Q$ themselves have exponent
$n$ in $G,$ as do pairwise products $xy\in G$ of elements of $Q$ (since
$y^{-1}\in Q,$ as it
belongs to the trajectory in $G$ determined by $x_0=1,$ $x_1=y).$

Can we conclude that under the above assumption
regarding Question~\ref{Q.fix_fix}, and for $Q$ translated as above
to contain $1,$ {\em all} elements of the subgroup of $G$ generated
by $Q$ have exponent $n,$ which would give an affirmative answer
to Question~\ref{Q.exp_n} for such $n$?
Not so far as I can see.
The products $xy$ mentioned above need not lie in $Q$
itself, so there is no evident reason why larger products,
e.g., $xyz$ for $x,y,z\in Q,$ should have exponent $n.$
(If we start with two elements $x,y\in Q,$ then an element of the
form $x^i y^j x^k$ {\em will} have exponent $n,$ since it
is conjugate in $G$ to $x^{i+k} y^j,$ and $x^{i+k}$ and
$y^j$ lie in $Q,$ being members of the trajectories
beginning $1,x$ and $1,y.$
But I see no reason why longer expressions in $x$ and $y,$
e.g., $[x,y]=x^{-1} y^{-1} x\,y,$ should have exponent $n.)$
The subgroup of $G$ generated by $Q$ will, of course, have a
universal exponent-$\!n\!$ homomorphic image; but some elements of $Q$
might fall together in that image.

(Groups subject to identities $x^n=1$ have also been used in
knot theory, \cite{Burnside1}, \cite{Burnside2}.)

I have not examined the consequences for $\r{Core}(G)$ of
any other identities on a group $G.$

\section{Counting generators of core quandles}\label{S.gen_rel}

Given a finitely generated group $G,$ what can be said about
the number of elements needed to generate the quandle $\r{Core}(G)$?
Here is a lower bound, which when $G$ is abelian
gives the precise value.

\begin{theorem}\label{T.gen_geq}
For $G$ a group, let $\r{gen}(G)$ denote
the minimum number of elements needed to generate $G$ as a group,
and for $Q$ an involutory quandle, let $\r{gen}(Q)$ denote
the minimum number of elements needed to generate $Q$ as a
quandle.

Then if $\r{gen}(G)$ is finite, and we write
$N$ for the subgroup of $G$ generated by the squares \textup{(}clearly
normal; so $G/N$ is the
universal exponent-$\!2\!$ homomorphic image of $G),$ then we have
\begin{equation}\begin{minipage}[c]{35pc}\label{d.fg}
$\r{gen}(\r{Core}(G))\ \geq
\ \r{max}(\r{gen}(G)\!+\!1,\,[G:N]).$
\end{minipage}\end{equation}

If $G$ is abelian, we have equality in~\eqref{d.fg}.
\end{theorem}

\begin{proof}
In view of Lemma~\ref{L.exp2}, the homomorphic image
$\r{Core}(G/N)$ of $\r{Core}(G)$ cannot be generated by
any proper subset of $G/N,$ hence requires
$[G:N]$ generators; hence $\r{Core}(G)$ itself
requires at least that many; so to get~\eqref{d.fg} it remains
to show that $\r{Core}(G)$ also requires at least
$\r{gen}(G)+1$ generators.

Suppose $\r{Core}(G)$ is generated by a set $S.$
Since $\r{Core}(G)$ is nonempty, $S$ must be nonempty;
choose $x\in S.$
Since translations under the group structure are
automorphisms of $\r{Core}(G),$
$\r{Core}(G)$ is also generated by $x^{-1}S;$
hence (since the $\!\lhd\!$-operation of
$\r{Core}(G)$ is a derived operation of $G),$
the group $G$ is generated by $x^{-1}S.$
But $1\in x^{-1}S;$ so $x^{-1}S-\{1\}$ also generates $G,$ so
$\r{card}(S)\geq\r{gen}(G)+1,$ as required.

To get the reverse inequality for abelian groups, let us first note
that if $G$ is an abelian group and $X$ any subset of $G$
containing $1,$ then an element $z\in G$ will belong to
the subquandle of $\r{Core}(G)$ generated by $X$ if and only if
\begin{equation}\begin{minipage}[c]{35pc}\label{d.gen_ab}
$z$ can be written as a product of powers of elements
of $X-\{1\},$ in which the exponents of all but at most one
of those elements are even.
\end{minipage}\end{equation}
Indeed, if we take an expression~\eqref{d.x_0dots_gp}
with all $x_i$ in $X,$
drop factors with $x_i=1,$ and combine the occurrences of
each element of $X,$
we get a product as described in~\eqref{d.gen_ab}, where the only member
of $X-\{1\}$ that can appear with odd exponent is the
$x_n$ of~\eqref{d.x_0dots_gp} if this is not $1.$
(If all the $x_i$ in~\eqref{d.x_0dots_gp} are~$1,$
we regard the resulting expression as an empty
product~\eqref{d.gen_ab}, which we understand to have value $1.)$

Given a finitely generated abelian group $G,$ we now
wish to construct a generating set for $\r{Core}(G)$
of the cardinality shown on the right-hand side of~\eqref{d.fg}.
To do this we will start with a set $X$ of possibly larger
cardinality, which it is easy
to verify generates $\r{Core}(G),$ then show how to ``combine''
certain pairs of elements of $X$ to get a set
$X'\subseteq\r{Core}(G)$ of the asserted cardinality,
whose closure under $\lhd$ contains $X,$ whence
$X'$ also generates $\r{Core}(G).$

To construct $X,$ let $\r{gen}(G)=n$
(no connection with the $n$ of~\eqref{d.x_0dots_gp}).
Being a finitely generated abelian group,
$G$ is a direct product of $n$ cyclic subgroups
$\langl g_i\mid g_i^{d_i}=1\rangl$ $(i=1,\dots,n)$
where each $g_i\in G,$ and each $d_i$ is either $0$ or $>1.$
Without loss of generality, assume $d_1,\dots,d_m$
even and $d_{m+1},\dots,d_n$ odd, where $0\leq m\leq n.$
Thus, the universal exponent-$\!2\!$ homomorphic image $G/N$
has order $2^m.$
Let
\begin{equation}\begin{minipage}[c]{35pc}\label{d.X=}
$X\ =\ X_0\ \cup\ X_1\ \cup\ X_2\ \cup\ X_3,$
\end{minipage}\end{equation}
where
\begin{equation}\begin{minipage}[c]{35pc}\label{d.X_0-3}
$X_0=\{1\},\ \ %
X_1=\{g_1,\dots,g_m\},\ \ %
X_2=\{\mbox{products of $\geq 2$ of $g_1,\dots,g_m\!$}\},\ \ %
X_3=\{g_{m+1},\dots,g_n\}.$ %
\end{minipage}\end{equation}

Note that $X_0\cup X_1\cup X_2$ is the set of all
products of subsets of $\{g_1,\dots,g_m\},$ and so has
cardinality $2^m,$ while $X_0\cup X_1\cup X_3$ has cardinality $n+1.$

To express an arbitrary $z\in G$ as in~\eqref{d.gen_ab} using
this $X,$ start with
the product of those $g_i$ with $1\leq i\leq m$ that occur with
odd exponent in $z,$ a member of $X_0\cup X_1\cup X_2;$ multiply
this element by appropriate {\em even} powers of $g_1,\dots,g_m\in X_1$
so as to achieve precisely the desired powers of each of those elements,
and, finally, note that each $g_i\in X_3$ has odd order,
hence the subgroup it generates is also generated by its square,
so that every power of $g_i$ can be regarded as an even power of
$g_i;$ so those $g_i$ can also be brought into our
product~\eqref{d.gen_ab} with even exponents so as to achieve
the desired value~$z.$
Hence $X$ indeed generates $\r{Core}(G).$

The trick for getting a generating set of smaller cardinality is
to combine elements of the sets $X_2$ and $X_3.$
Given $g_{i_1}\dots g_{i_k}\in X_2$
and $g_j\in X_3,$
I claim that using the operation $\lhd,$ we can obtain
these two elements of $X$ from their product
\begin{equation}\begin{minipage}[c]{35pc}\label{d.y=}
$y\ =\ g_{i_1}\dots g_{i_k}\,g_j$
\end{minipage}\end{equation}
and the members of $X_0\cup X_1.$
Indeed, first note that $y^2 g_{i_1}^{-2}\dots g_{i_k}^{-2}=g_j^2$
will lie in the subquandle generated by these
elements (cf.~\eqref{d.gen_ab},
noting that $1,\,g_{i_1},\dots,g_{i_k}\in X_0\cup X_1).$
Since $g_j$ has odd order, some power of $g_j^2$
gives us $g_j,$ as desired.
Multiplying $y$ by an even power of this element $g_j$
which equals $g_j^{-1},$ we get $g_{i_1}\dots g_{i_k},$ the
other element we wanted to recover.

By combining in this way
pairs of elements, one from $X_2$ and one from $X_3,$
until all elements of one of these
sets have been used, we can replace $X_2\cup X_3$ in our generating
set for $\r{Core}(G)$ by a set $X_{2,3},$ whose cardinality
is the greater of the cardinalities of $X_2$ and $X_3.$
(Namely, $X_{2,3}$ will consist of the products
$g_{i_1}\dots g_{i_k}\,g_j$
that we have introduced, together with the unused elements, if any,
of one of $X_2$ and $X_3.)$
The cardinality of the resulting generating set,
$X_0\cup X_1\cup X_{2,3},$
will be the greater of the cardinalities of
$X_0\cup X_1\cup X_2$ and $X_0\cup X_1\cup X_3,$ which are
$2^m=[G:N]$ and $n+1=\r{gen}(G)+1$
respectively, giving equality in~\eqref{d.fg}, as desired.
\end{proof}

What about an upper bound for $\r{gen}(\r{Core}(G))$ for
a not necessarily abelian group $G$?
Can we even expect the core quandle of, say, a free group on more than
one generator to be finitely generated?
At first sight it seems implausible that
the {\em symmetric} expressions~\eqref{d.x_0dots_gp}
in the elements of some finite set $X$ should be able to
represent {\em arbitrary} elements of $G,$ which need not
have any sort of symmetry -- unless, perhaps, we can somehow
arrange that most of the terms on, say, the right sides
of our expressions~\eqref{d.x_0dots_gp} cancel one another,
while the left sides carry the structure of our elements.

Surprisingly, we can do this.
The key idea is that the distinction between free abelian
groups and free groups concerns commutators,
and that if for every pair of generators $g_i,$ $g_j$ $(i<j)$ of
our free group, we include in the set with which we plan
to generate $\r{Core}(G)$
the elements $g_i,$ $g_j$ and $g_i\,g_j,$ then
$g_i^{-1},$ $g_j^{-1}$ and $g_i\,g_j,$ multiplied
in one order, give the commutator
$[g_i,g_j],$ while multiplied in the reverse order, they give $1:$
\begin{equation}\begin{minipage}[c]{35pc}\label{d.[]vs1}
$g_i^{-1}\,g_j^{-1}\,(g_i\,g_j)\ =\ [g_i,g_j],$\quad
$(g_i\,g_j)\ g_j^{-1}\,g_i^{-1}\ =\ 1.$
\end{minipage}\end{equation}

I will describe below how to use this fact to get a generating set
of cardinality $2^n$ for the core of a free group on $n$ generators,
then show in Theorem~\ref{T.gen_leq} how to improve that bound somewhat
for more general finitely generated groups.

Let $G$ be the free group on generators $g_1,\dots,g_n,$ and
let $X$ be the set of all $2^n$ products $g_{i_1}\dots g_{i_r}$
with $0\leq r\leq n$ and $i_1<\ldots<i_r.$
(In particular, $X$ contains the empty product, $1.)$

Given $z\in G,$ we wish to find an expression~\eqref{d.x_0dots_gp}
with all $x_i$ in $X,$ which has in $G$ the value~$z.$

Since $1\in X,$ we can, by using $1$ for various $x_i,$
represent in the form~\eqref{d.x_0dots_gp} any symmetric
string of elements of $X$ with exponents $\pm 1$
(with no restriction that these exponents alternate
between $+1$ and~$-1).$

Let $G'$ be the commutator subgroup of $G,$ so that
$G/G'$ is free abelian on the images of $g_1,\dots,g_n.$
As in the proof of Theorem~\ref{T.gen_geq}, we can find a word $w_0$
of the form~\eqref{d.x_0dots_gp}
in the elements of $X$ which, evaluated in $G/G',$ agrees with $z.$
Thus if, instead, we evaluate $w_0$ in $G,$ it gives
an element $z_0$ which is congruent to $z$ modulo $G'.$
Say $z = u\,z_0$ with $u\in G'.$

The group $G'$ is generated by conjugates in $G$ of
elements $[g_i,\,g_j]$ with $1\leq i<j\leq n.$
Each such conjugate will be the value of an expression
$v\,g_i^{-1} g_j^{-1} (g_i\,g_j)\,v^{-1},$
where $v$ is an expression in the elements of $X,$
and by $v^{-1}$ we mean the expression gotten by reversing the order
of factors and interchanging exponents $+1$ and $-1.$
Let $w_1$ denote a word gotten by multiplying together a
family of such expressions for conjugates of commutators,
and inverses of such expressions,
which, when evaluated in $G,$ gives~$u.$

Now let $\overline{w_1}$ denote the word gotten by reversing
the order of the terms from $X$ appearing in $w_1$
(without changing the exponents $+1$ and $-1),$
and take $w= w_1\,w_0\,\overline{w_1}\,.$
Since $w_0$ was symmetric, $w$ will be symmetric, hence when
evaluated in $G,$ it gives
a member of the subquandle of $\r{Core}(G)$ generated by $X.$
Moreover, as noted earlier, our expressions in $w_1$ for commutators,
when reversed in $\overline{w_1},$ give expressions which,
evaluated in $G,$ give $1,$ hence the same is true for the formal
conjugates $v\,g_i^{-1} g_j^{-1} (g_i\,g_j)\,v^{-1}$
we used, hence $\overline{w_1}$ itself,
evaluated in $G,$ gives $1.$
Hence $w= w_1\,w_0\,\overline{w_1},$
evaluated in $G,$ gives $u\,z_0\,1 = z,$ as desired.
So $X$ indeed generates $\r{Core}(G).$

The next result records the consequence of the above bound,
then notes how it can be strengthened.

\begin{theorem}\label{T.gen_leq}
If $G$ is a finitely generated group with $\r{gen}(G)=n,$ then
\begin{equation}\begin{minipage}[c]{35pc}\label{d.leq_2^n}
$\r{gen}(\r{Core}(G))\ \leq\ 2^n.$
\end{minipage}\end{equation}

More sharply, if $N$ is the subgroup of $G$ generated
by all squares \textup{(}a normal subgroup\textup{)},
and $\r{gen}(G/N)=m\leq n,$ we have
\begin{equation}\begin{minipage}[c]{35pc}\label{d.leq_1+n+}
$\r{gen}(\r{Core}(G))\ \leq\ (1+n+n(n{-}1)/2)+
2^m - (1 + m + m(m{-}1)/2).$
\end{minipage}\end{equation}

Namely, if we take a generating set $\{g_1,\dots,g_n\}$
for the group $G$ such that the images in $G/N$ of $g_1,\dots,g_m$
generate $G/N,$ while $g_{m+1},\dots,g_n\in N,$ then a generating
set $X$ for $\r{Core}(G)$ with the above cardinality
is given by the set of those products $g_{i_1}\dots\,g_{i_k}$
\textup{(}including the empty product, with $k=0)$
such that $i_1<\dots<i_k,$ and either $i_k\leq m,$ or $k\leq 2.$
\end{theorem}

\begin{proof}
Above, we established~\eqref{d.leq_2^n} under the
simplifying assumption that $G$ was the free group on $n$ generators.
Since any $\!n\!$-generator group is a homomorphic image
of such a free group, the same bound
holds for all $\!n\!$-generator groups.

To get the sharper bound~\eqref{d.leq_1+n+}, note first that
starting with any $\!n\!$-element generating set for $G,$ we can
index it so that the images of $g_1,\dots,g_m$ generate $G/N,$
then modify each of $g_{m+1},\dots,g_n$ by a product of terms
$g_1,\dots,g_m$ so that the
new $g_{m+1},\dots,g_n$ all have trivial image in $G/N,$
as in the last sentence of the theorem.

Now let $X$ be the set of products of elements of $\{g_1,\dots,g_n\}$
described in that same sentence.
To count the elements of $X,$ note that
the numbers of such products of $0,$ $1$ and $2$ factors from
$\{g_1,\dots,g_n\}$ are respectively $1,$ $n,$ and
$n(n-1)/2,$ while the set of products arising from arbitrary subsets
of $\{g_1,\dots,g_m\}$ has $2^m$ elements.
These two sets intersect in the set of
products of $0,$ $1$ and $2$ factors from
$\{g_1,\dots,g_m\},$ which has cardinality $1+m+m(m-1)/2,$
which we therefore subtract off; so the cardinality of $X$
is indeed the right-hand side of~\eqref{d.leq_1+n+}.

Given any $z\in G,$ let us now describe how to represent it as a
symmetric expression in the elements of $X.$
We choose the middle term of the expression to be a member of $X$
having the same image in $G/N$ as $z$ has.
As in the proof of Theorem~\ref{T.gen_geq}
(but without the complication of splitting apart products
$g_{i_1}\dots g_{i_k}\,g_j),$ we surround that term
symmetrically with terms $g_1^{\pm 1},\dots,g_n^{\pm 1},$
so that the result has the same image in the
abelian group $G/G'$ as $z.$
Finally, as in the discussion of the
case of free $G,$ we surround
the resulting expression with expressions which, on the
left-hand side, give a product of conjugates of commutators
$[g_i,\,g_j]$ $(1\leq i<j\leq n)$ which brings our
expression to exactly the value $z,$ while on the right-hand
side, they reduce to~$1.$
We have thus written $z$ as a $\!\lhd\!$-expression in elements of $X.$
\end{proof}

Note that if $m=n$ above, then the lower bound of
Theorem~\ref{T.gen_geq} and the upper bound of the above theorem
agree, and we get the exact result $\r{gen}(\r{Core}(G))=2^n.$

In the opposite direction, if $m\leq 2,$ no subsets
of $\{g_1,\dots,g_m\}$ have $>2$ elements, so
the upper bound of~\eqref{d.leq_1+n+} simplifies to $1+n+n(n-1)/2;$
but this is in general larger than the lower bound of~\eqref{d.fg}.

We have seen that the lower bound of~\eqref{d.fg}
is witnessed by abelian groups, but we ask

\begin{question}\label{Q.which}
Can the upper bound of~\eqref{d.leq_1+n+} be improved?
\end{question}

Of course, that bound can be strengthened for groups satisfying
additional conditions.
For instance, if in some group $G$ with generators
$g_1,\dots,g_n$ as shown, we know that a
certain commutator $[g_i,g_j]$ is $1,$ or more generally,
is expressible as a product of conjugates of other
commutators, and $j>m,$ then the generator $g_i g_j,$
no longer needed to get an expression for $[g_i,g_j]$
as in~\eqref{d.[]vs1}, can be dropped from our set~$X.$

I have not examined

\begin{question}\label{Q.fp}
If a group $G$ is {\em finitely presented}, is the same true
of the involutory quandle $\r{Core}(G)$?
If so, what bound can be put on the number of relators
needed by a presentation of $\r{Core}(G),$ in terms of the
numbers of generators and relators in a presentation of~$G$?
\end{question}

We end this note with a few tangential observations.

\section{Comparison with heaps}\label{S.heap}

A derived operation on groups related
to the core quandle operation $\lhd$ is the ternary operation
\begin{equation}\begin{minipage}[c]{35pc}\label{d.heap}
$\tau(x,\,y,\,z)\ =\ x\,y^{-1}\,z,$
\end{minipage}\end{equation}
which satisfies the identities
\begin{equation}\begin{minipage}[c]{35pc}\label{d.heap_assoc}
$\tau(\tau(v,\,w,\,x),\,y,\,z)\ =
\ \tau(v,\,\tau(y,\,x,\,w),\,z)\ =
\ \tau(v,\,w,\,\tau(x,\,y,\,z)),$
\end{minipage}\end{equation}
\begin{equation}\begin{minipage}[c]{35pc}\label{d.heap_inv}
$\tau(x,\,x,\,y)\ =\ y\ = \tau(y,\,x,\,x).$
\end{minipage}\end{equation}
A set with an operation $\tau$ satisfying~\eqref{d.heap_assoc}
and~\eqref{d.heap_inv} is called a {\em heap}.
(See \cite[Exercises~9.6:10-11]{245} for some background and
references.)
For $G$ a group, let us write $\r{Heap}(G)$ for the heap with
the same underlying set as $G,$ and the operation~\eqref{d.heap}.

As with involutory quandle structures, the heap structure on
$\r{Heap}(G)$ does not determine the group structure:
again, every right or left translation operation of the
group structure is an automorphism of the heap structure.
But in contrast to the case of involutory quandles,
every heap structure on a nonempty set $H$ does arise as above from
a group structure on $H,$
which is unique {\em up to isomorphism,} and which
becomes unique as soon as one chooses an element $e\in H$ to
be the identity element.
The group structure is then given by
\begin{equation}\begin{minipage}[c]{35pc}\label{d.gp_fr_hp}
$x\,y\ =\ \tau(x,\,e,\,y),\quad x^{-1}\ =\ \tau(e,\,x,\,e).$
\end{minipage}\end{equation}

Because of this near-equivalence with groups, heaps
are not much studied for their own sake, though one sometimes
calls on the concept in situations where a natural heap structure
exists but a natural group structure does not.
Namely, given two objects
$C$ and $D$ of a category, the set of isomorphisms $C\to D$
has only a natural structure of heap, given by the same
formula~\eqref{d.heap}.

The core quandle structure on the underlying set of a group is,
clearly, expressible in terms of the heap structure:
\begin{equation}\begin{minipage}[c]{35pc}\label{d.qndl_fr_heap}
$x\lhd y\ =\ \tau(x,\,y,\,x).$
\end{minipage}\end{equation}
But this loses much more information about the group than the heap
structure did.
As we have seen, not every involutory quandle
arises from a group, or is even
embeddable in one arising in that way, and when it does
arise from a group, it need not determine that group up to isomorphism
(Lemma~\ref{L.cm_ncm}).

\section{A class of structures weaker than involutory quandle
structures}\label{S.thoughts}

In this note, special behavior has repeatedly involved the exponent~$2$
in groups (e.g., Lemma~\ref{L.exp2}, Proposition~\ref{P.fin_traj}(ii),
the second paragraphs of Proposition~\ref{P.rep}
and Theorem~\ref{T.gen_geq}, and Theorem~\ref{T.gen_leq}).
A generalization of the subject,
in which more exponents can be expected to show such
behavior, would be to study, for general $n>1,$ the binary
operator $\lhd_{n}$ on underlying sets of groups defined
to carry the terms $x_1$ and $x_0$ of
a trajectory to $x_n;$ in other words,
\begin{equation}\begin{minipage}[c]{35pc}\label{d.lhd_n}
$x\,\lhd_{n}\,y\ =\ x\,(y^{-1}x)^{n-1}.$
\end{minipage}\end{equation}
The operation we have called $\lhd$ is in this notation $\lhd_{2}.$
For $n>2,$ $\lhd_n$ is not, in general, a quandle operation.

If $(x_i)_{i\in\Z}$ is a trajectory in a group,
and $S$ any subset of $\Z,$
then it is not hard to show that the set of terms generated
under $\lhd_{n}$ by $\{x_i\mid i\in S\}$
will have the property that each of its members is
$x_j$ for some $j$ which is both congruent modulo $n$ to
some member of $S$ and congruent modulo $n-1$ to some
(possibly different) member of $S.$
Note also that the right-hand side of~\eqref{d.lhd_n} has value $x$
if and only if $(y^{-1}x)^{n-1}=1,$ and value $y$
if and only if $(y^{-1}x)^n=1.$
So it seems that the $\!\lhd_{n}\!$-analogs of core quandles of
groups will show interesting behavior involving exponents
that divide $n$ or $n-1.$

\section{On language and notation}\label{S.flip}

When I first looked at the operation $x\,y^{-1}x$ on groups,
and the identities it satisfies, not knowing
that these had already been studied,
I wrote a draft of this note in which a set with
an operation satisfying those identities was called a
``flip-set'', since that operation can be looked at
as flipping $y$ past $x$ in the trajectory they generate.
After learning that such structures had already been
studied, I brought this note into conformity with standard language.
However, I find ``involutory quandle'' cumbersome
compared with ``flip-set''.
I leave it to workers more involved in the subject
to decide whether it might be worth switching to a name
such as ``flip-quandle''.

The notation I originally used for $x\,y^{-1}x$ was $x\,\natural\,y$
(which I read ``$x$ flip $y$'' -- I don't know
how $x \lhd y$ is pronounced).
It might in some contexts be convenient to distinguish the operations
of $\r{Conj}(G)$ and $\r{Core}(G)$ as $\lhd$ and $\natural\,.$
(The quandle operations $x^n\,y\,x^{-n},$ and the non-quandle
operations~\eqref{d.lhd_n} discussed above could then
be distinguished as $\lhd_n$ and $\natural_n.)$

\section{Acknowledgements}\label{S.Ackn}
I am indebted to Yves de Cornulier and Ualbai Umirbaev
for pointing out that the objects I was calling flip-sets are known as
involutory quandles, to Valeriy Bardakov,
J. Scott Carter, and J\'{o}zef H. Przytycki
for pointing me to related material on the subject,
and to the referee for several helpful suggestions.

\end{document}